\documentclass[12 pt]{amsart}
\usepackage{amscd,amsmath,amsthm,amssymb}
\usepackage{verbatim}
\usepackage[all]{xy}
\usepackage{color}
\usepackage{hyperref}
\usepackage{multirow}
\usepackage{subcaption}

\usepackage{tikz, float} \usetikzlibrary {positioning,calc,intersections}

\usepackage{enumitem}

\def\NZQ{\mathbb}               

\def\FF{{\NZQ F}}

\def\frk{\mathfrak}               

\def\Phi{{\frk n}}
\def\Phi{{\frk N}}

\def\kb{{\mathbf k}}

\def\xb{{\mathbf x}}
\def\yb{{\mathbf y}}
\def\zb{{\mathbf z}}
\def\pb{{\mathbf p}}

\def\A{{\mathcal A}}

\def\Fc{{\mathcal F}}

\def\Sc{{\mathcal S}}

\def\Fc{{\mathcal F}}

\def\Sc{{\mathcal S}}

\def\xb{{\mathbf x}}
\def\yb{{\mathbf y}}
\def\zb{{\mathbf z}}

\def\opn#1#2{\def#1{\operatorname{#2}}} 
\opn\chara{char} \opn\length{\ell} \opn\pd{pd} \opn\rk{rk}
\opn\projdim{proj\,dim} \opn\injdim{inj\,dim} \opn\rank{rank}
\opn\depth{depth} \opn\grade{grade} \opn\height{height}
\opn\embdim{emb\,dim} \opn\codim{codim}

\opn\Tr{Tr} \opn\bigrank{big\,rank}
\opn\superheight{superheight}\opn\lcm{lcm}
\opn\trdeg{tr\,deg}
\opn\reg{reg} \opn\lreg{lreg} \opn\ini{in} \opn\lpd{lpd}
\opn\size{size} \opn\sdepth{sdepth}
\opn\link{link}\opn\fdepth{fdepth}\opn\lex{lex}
\opn\LM{LM}
\opn\LC{LC}
\opn\NF{NF}
\opn\Merge{Merge}
\opn\sgn{sgn}
\opn\suppPos{suppPos}
\opn\div{div} \opn\Div{Div} \opn\cl{cl} \opn\Pic{Pic}
\opn\Prin{Prin}
\opn\op{op}
\opn\indeg{indeg} \opn\outdeg{outdeg}
\opn\red{red}
\opn\Spec{Spec} \opn\Supp{Supp} \opn\supp{supp} \opn\Sing{Sing}
\opn\Ass{Ass} \opn\Min{Min}\opn\Mon{Mon}
\opn\Ann{Ann} \opn\Rad{Rad} \opn\Soc{Soc}
 \opn\Ker{Ker} \opn\Coker{Coker} \opn\Am{Am}
\opn\Hom{Hom} \opn\Tor{Tor} \opn\Ext{Ext} \opn\End{End}
\opn\Aut{Aut} \opn\id{id}

\opn\nat{nat}
\opn\pff{pf}
\opn\Pf{Pf} \opn\GL{GL} \opn\SL{SL} \opn\mod{mod} \opn\ord{ord}
\opn\Gin{Gin} \opn\Hilb{Hilb}\opn\sort{sort}
\opn\Image{Image}
\opn\aff{aff} \opn\con{conv} \opn\relint{relint} \opn\st{st}
\opn\lk{lk} \opn\cn{cn} \opn\core{core} \opn\vol{vol}
\opn\link{link} \opn\star{star}\opn\lex{lex}\opn\set{set}
\opn\dist{dist}
\opn\gr{gr}

\def\pot#1#2{#1[\kern-0.28ex[#2]\kern-0.28ex]}

\opn\dirlim{\underrightarrow{\lim}}
\opn\inivlim{\underleftarrow{\lim}}
\def\Implies{\ifmmode\Longrightarrow \else
        \unskip${}\Longrightarrow{}$\ignorespaces\fi}
\def\implies{\ifmmode\Rightarrow \else
        \unskip${}\Rightarrow{}$\ignorespaces\fi}
\def\iff{\ifmmode\Longleftrightarrow \else
        \unskip${}\Longleftrightarrow{}$\ignorespaces\fi}

\let\:=\colon
\newtheorem{Theorem}{Theorem}[section]
\newtheorem{Lemma}[Theorem]{Lemma}

\newtheorem{Proposition}[Theorem]{Proposition}

\theoremstyle{remark}
\newtheorem{Remark}[Theorem]{Remark}

\theoremstyle{definition}
\newtheorem{Example}[Theorem]{Example}

\newtheorem{Definition}[Theorem]{Definition}

\let\kappa=\varkappa
\def\qed{\ifhmode\textqed\fi
      \ifmmode\ifinner\quad\qedsymbol\else\dispqed\fi\fi}
\def\textqed{\unskip\nobreak\penalty50
       \hskip2em\hbox{}\nobreak\hfil\qedsymbol
       \parfillskip=0pt \finalhyphendemerits=0}
\def\dispqed{\rlap{\qquad\qedsymbol}}

\opn\dis{dis}
\def\pnt{{\raise0.5mm\hbox{\large\bf.}}}

\opn\Lex{Lex}
\opn\syz{{\rm syz}}
\opn\spoly{{\rm spoly}}
\opn\LM{{\rm LM}}
\opn\lm{{\rm lm}}
\opn\MVT{{\rm MVT}}
\opn\projdim{{\rm projdim}}
\opn\lcm{{\rm lcm}} \opn\A{\mathcal A}

\opn\prob{{\rm prob}}

\numberwithin{equation}{section}

\tikzstyle{Cwhite}=[scale = .6,circle, fill = white, minimum size=2.5mm]
\tikzstyle{Cgray}=[scale = .4,circle, fill = gray, minimum size=3mm]
\tikzstyle{Cblack2}=[scale = .4,circle, fill = black, minimum size=3mm]
\tikzstyle{Cblack}=[scale = .7,circle, fill = black, minimum size=3mm]
\tikzstyle{C0}=[scale = .9,circle, fill = black!0, inner sep = 0pt, minimum size=3mm]
\tikzstyle{C1}=[scale = .7,circle, fill = black!0, inner sep = 0pt, minimum size=3mm]
\tikzstyle{Cred}=[scale = .4,circle, fill = red, minimum size=3mm]
\tikzstyle{Cblue}=[scale = .4,circle, fill =blue, minimum size=3mm]

\usepackage{nomencl}
\usepackage{etoolbox}

\makenomenclature

\renewcommand\nomgroup[1]{%
  \item[\bfseries
  \ifstrequal{#1}{A}{Algebraic reliability of multi-state systems}{%
  \ifstrequal{#1}{B}{Simple Multi-state $k$-out-of-$n$ systems}{%
  \ifstrequal{#1}{C}{Generalized multi-state $k$-out.of-$n$ systems}{%
  \ifstrequal{#1}{D}{Binary multi-state $k$-out-of-$n$ systems with multi-state components}}}}%
]}

\begin{document}

\title{Algebraic reliability of multi-state $k$-out-of-$n$ systems}

\author{Patricia Pascual-Ortigosa} 
\address{Departamento de Matem\'aticas y Computaci\'on,  Universidad de La Rioja, Spain}
\email{patricia.pascualo@unirioja.es}

\author{Eduardo S\'aenz-de-Cabez\'on}
\address{Departamento de Matem\'aticas y Computaci\'on,  Universidad de La Rioja, Spain}
\email{eduardo.saenz-de-cabezon@unirioja.es}

\author{Henry P. Wynn}
\address{Department of Statistics, London School of Economics, UK}
\email{h.wynn@lse.ac.uk}


\maketitle

\begin{abstract}
	In this paper we review different definitions that multi-state $k$-out-of-$n$ systems have received along the literature and study them in a unified way using the algebra of monomial ideals. We thus obtain formulas and algorithms to compute their reliability and bounds for it. We provide formulas and computer experiments for simple and generalized multi-state $k$-out-of-$n$ systems and for binary $k$-out-of-$n$ systems with multi-state components.\end{abstract}

\section{Introduction}\label{Sec:Intro}
We say that a system is a $k$-out-of-$n$:G system (G for good) if it works whenever $k$ of its $n$ components work, and that it is a $k$-out-of-$n$:F (F for fail) if it fails whenever $k$ of its $n$ components fail. $k$-out-of-$n$ systems are one of the most relevant types of systems studied in reliability theory due to their theoretical interest and wide range of applications, cf. \cite{HZW00b,KZ03, E10}. The multi-state version, which can model more general situations, has been object of intense research in the last decades and is also applied in a variety of situations \cite{HZF03,HZW00,ADMX15,E18,R19}. Since the first definition of multi-state $k$-out-of-$n$ systems \cite{EPS78} several authors have proposed different definitions and generalizations, together with particular methods to evaluate the reliability of these systems, see for instance \cite{BK94,HZF03,HAR07,DZLL10,DZTL10,AEHR11,ADMX15,MXCS17} and references therein.

We list a number of examples of this kind of sytems. 
\begin{enumerate}
\item {\em Power generation}. The safety and reliability of power systems is an essential component of energy security and is increasing its importance in a period in which there are likely to be radical changes in energy supply as governments adopt zero net carbon strategies and use more renewable sources, such as wind power, which may be more volatile. There are four standard states of generation for an energy unit: (i) available and in service, (ii) available and not in  service (iii) planned outage, (iv) unplanned outage. Considering that a national electricity grid will have many sources of supply and different components will be in different states, this represents a challenging multi-component multi-state network. There is also a strong time aspect leading to strict definitions such as FOR: Forced Outage Rate and AV: Availability, which form part of supply contracts and regulation. Recent books are \cite{RT19,SJM19,LPLLCS19} for a $k$-out-of-$n$ approach.

\item {\em DNA repair}. DNA damage is a biological process  that can upset important functions such as replication. DNA damage is different from mutation, although both occur.  The system can be in very many states, depending on the amount of cell loss of different types. Areas of study include the fundamental equilibria between repair and damage, needed to sustain the systems. Initial  models make assumptions, similar to those in reliability,  for example, that occurrence at break sites happen independently \cite{DZNHMKS15}.

\item  {\em Software reliability and Bayes nets}. It is natural in several areas of reliability to take a probabilistic state-space approach. This is particularly true of one of the main traditions of software reliability and provides an alternative to rule based formal methods. An advantage of this approach is that it can model systems as a Bayes net and link up with modern theories of causation. Also important in such systems is the idea of {\em degradation} which automatically implies different levels of reliability and is particularly important in the analysis of safety critical systems; see \cite{FB14} for a comprehensive approach.
\end{enumerate}

In the failure of $k$-out-of-$n$ components, the number $k$ is a simple metric to describe degradation (mentioned above) and this extends to the multi-state methods addressed here.  A useful way to think of the latter is that there is a damage ``frontier" beyond which the system is deamed to have failed or to have reached a level, for example, at which the unit may be switched off for maintenance. This may be planned or unplanned (as mentioned above for power generation).  Another way of conceptualizing these issues is that the metric $k$ is simply a way of {\em counting} some (bad) aspect of the system and counting is surely a basic combinatorial and algebraic activity. Broadly, research on the theory of  $k$-out-of-$n$ methods divides into (i) combinatorial and algebraic theory, as in this paper, and (ii) simulations studies, which are typically of a Markovian type. For the combinatorial methods generating functions play an important role \cite{Y06}. In our work this is reflected in the use of Hilbert series, which are essentially a type of generating function.  For sequential $k$-out-of-$n$  problems one  often converts the system into a Markov chain, inspect the ergodic behaviour and benchmark against probabilistic asymptotics from large deviation theory and boundary crossing methods. A main tool is that of de Bruijn graphs which track the change of a moving window between time steps \cite{LD09} .  Signature analysis has also been applied to $k$-out-of-$n$ systems \cite{M19}. The methods employ the inherent symmetries in the order statistics of failure events to simplify reliability bounds, \cite{KS17}. Genome analysis is one science that makes much use of a type of  $k$-out-of-$n$ analysis under a heading of {\em k-mer}: the detection of special genome sequence of length $k$ out of a much longer sequence, \cite{RLDC13}. There is a dominance of computer base search methods in the area and some also use the de Bruijn graph methods. The idea of a ``special sequence" makes the field quite close to percolation theory where the sequence is a percolation through a lattice structure of some kind.  

The algebraic method for the analysis of system reliability associates a monomial ideal to a coherent system and by studying algebraic properties of this ideal obtains information about the system and its reliability \cite{SW10,SW11,SW12, SW15}, see Appendix \ref{appendix:ideals} for a basic introduction to this method. The principal objective is to obtain general extensions of classical Bonferroni bounds in multi-state system reliability.  It is a general method that can be adapted to different kind of systems, both binary and multi-state. In this paper we review the different definitions of multi-state $k$-out-of-$n$ systems, study them in an algebraic way, and apply the algebraic method as a unified way to compute their reliability. The foundation has two  parts: a description of the system, including the idea of a state, and the stochastic model which defines the occupancy of the state. The next step is to map the system into an algebraic object called a monomial ideal, which can be handled via combinatorial algebra, including the use of computer algebra (already well developed for this purpose). The compact inclusion-exclusion formulae needed for the bounds start by being distribution-free and require special Betti numbers which are attached to the "live" terms in the formulae. For simple probability models it is then straightforward to obtain the actual probability bounds.

A problem for the reliability computation of these systems is the computational burden when complexity increases. Several algorithms have been proposed to compute the exact reliability of these systems, see \cite{AZD09,CBAZ12,ZC10,TZY08,MXAB15}; also, Ding et al. propose in \cite{DZLL10} a framework for reliability approximation. Our approach, while enumerative, shows good performance and can provide both exact reliability and bounds in the case of i.i.d components and in the case of independent non-identical components.

The outline of the paper is the following: in Section \ref{Sec:Prelim} we give a quick overview of the algebraic method for system reliability analysis, in particular when applied to multi-state systems. In Section \ref{Sec:FirstMS_KN} we show the first definitions of multi-state $k$-out-of-$n$ systems, give an algebraic version of them and use it to analyse the reliability of this kind of systems. In Section \ref{Sec:GeneralizedKN} we study generalized multi-state $k$-out-of-$n$ systems and in Section \ref{Sec:B_MS_KN} we focus on a type of binary $k$-out-of-$n$ systems with multi-state components and give an example of application of these systems. A simple storage problem is used for illustration. 

\section*{Nomenclature}

\begin{description}[font=\sffamily, leftmargin=3cm]
\item[$S$] {A coherent system}
\item[{$n$}] {Number of components of the system $S$}

\item[{$m$} ] {Maximum level of performance of the system $S$}

\item[{$\mathcal{S}=\lbrace 0,\dots,m\rbrace$}] {Possible states of the system $S$}

\item[$c_i$]  {Component $i$ of the system, $i\in\lbrace1,\dots,n\rbrace$}

\item[$m_i$]  {Maximum level of performance of the component $c_i$, $i\in \lbrace 1,\dots,n\rbrace$}

\item[$\mathcal{S}_i=\lbrace0,\dots,m_i\rbrace$]  {Possible states of the component $c_i$, $i\in \lbrace1,\dots,n\rbrace$}

\item[{$\phi:\mathcal{S}_1\times\dots\times\mathcal{S}_n\rightarrow\mathcal{S}$}] {Structure function of the system $S$}

\item[{$\xb=(x_1,\dots,x_n)$}]  {Vector of components' states}

\item[{$\mathcal{F}_{S,j}$} ] {Set of $j$-working states of $S$}

\item[{$\overline{\mathcal{F}}_{S,j}$} ] {Set of minimal $j$-working states of $S$}

\item[{$I_{S,j}$} ] {$j$-reliability ideal of $S$}

\item[{$G(I_{S,j})$}]  {Unique minimal monomial generating set of $I_{S,j}$}

\item[{$H_{I_{S,j}}$}]  {Numerator of the Hilbert series of $I_{S,j}$}

\item[{$\beta_i(I), \beta_{i,j}(I)$ } ] {Betti numbers and graded Betti numbers of $I$}

\item[{$I_{(k,n),j}$}]  {$j$-reliability ideal of a simple multi-state $k$-out-of-$n$ system}

\item[{$N_j$}] {Number of components in state $j$ or above, $j\in\lbrace 1,\dots,M\rbrace$}

\item[{$S_{n,(k_1,\dots,k_M)}$}] {Generalized multi-state $k$-out-of-$n$ system}

\item[{$I_{n,(k_1,\dots,k_M)}$}] {$j$-reliability ideal of a generalized multi-state $k$-out-of-$n$ system}

\item[{$p_{i,j}$}] {Probability that the component $i$ is in level greater than or equal to $j$}

\item[{$R_{S,j}$}] {Probability that the system $S$ is performing at level greater than or equal to $j$}

\item[{$r_{S,j}$}] {Probability that the system $S$ is performing at level $j$}

\item[{$S_{m,n,k}$}] {$m$-multi-state $k$-out-of-$n$:G system}

\item[{$J^m_{[n,k]}$}] {$j$-reliability ideal of the system $S_{m,n,k}$}

\item[{$N^m_{[n,k]}$}] {Number of generators of the ideal $J^m_{[n,k]}$}

\end{description}

\section{Algebraic reliability of multi-state systems}\label{Sec:Prelim}

Let $S$ be a system with $n$ components that can be in any of a set of $m+1$ possible states $\Sc=\{0,\dots,m\}$. Each component $c_i$ of  $S$ can be in a discrete number of ordered states $\Sc_i=\{0,\dots,m_i\}$. The states of the system are also ordered and measure the overall performance of the system. We assume that state $j$ represents better performance than state $i$ whenever $j>i$. We define a structure function $\phi$ that for each $n$-tuple of component states outputs the state of the system i.e. $\phi:\Sc_1\times \cdots \times \Sc_n\rightarrow \Sc$. We say that the system is {\it coherent} if $\phi(\xb)\geq\phi(\yb)$ whenever $\xb>\yb$, which means that the component states given by $\xb$ are greater or equal than those given by $\yb$ and there is at least one improvement. Conversely, $\phi(\xb)\leq\phi(\yb)$ whenever $\xb<\yb$. If $m_1=\dots=m_n=1$, then we say that the system has {\it binary components}. If $m=1$, then we say that the system is itself {\it binary}. We have therefore the following types of systems with respect to their number of states:
\begin{itemize}
\item[-] If $m=1$ and $m_i=1$ for all $i$, we have a binary system with binary components. These are usually simply referred to as {\em binary systems}.
\item[-] If $m>1$ and $m_i=1$ for all $i$, we have a multi-state system with binary components.
\item[-] If $m=1$ and there is at least one component $i$ such that $m_i>1$, we have a binary system with multi-state components.
\item[-] If $m>1$ and there is at least one component $i$ such that $m_i>1$, we have a multi-state system with multi-state components.
\end{itemize}
We basically follow here the notation in \cite{GN17} and \cite{N11} but we allow a more general kind of systems, since we do not restrict to the case that $\max(\Sc)\leq\max(\Sc_i)\,\forall i$. For other definitions of multi-state systems and a review of multi-state reliability analysis, we refer to \cite{EPS78, LL03,YJ12, LFK17} and the references therein.


Let $S$ be a coherent system with $n$ components and let $\Fc_{S,j}$ be the set of tuples of components' states $\xb$ such that $\phi(\xb)\geq j$ for some $0<j\leq m$. The elements of $\Fc_{S,j}$ are called {\it $j$-working states} of $S$. Let $\overline{\Fc}_{S,j}$ be the set of minimal $j$-working states, i.e. states in $\Fc_{S,j}$ such that the degradation of the performance of any component provokes that the overall performance of the system is degraded to some $j'<j$. Let now $R=\kb[x_1,\dots,x_n]$ be a polynomial ring over a field $\kb$. Each tuple of components' states $(s_1,\dots,s_n)\in \Sc_1\times\cdots \times \Sc_n$ corresponds to the monomial  $x_1^{s_1}\cdots x_n^{s_n}$ in $R$. The {\it coherence property} of the system is equivalent to saying that the elements of $\Fc_{S,j}$ correspond to the monomials in an ideal, denoted by $I_{S,j}$ and called the {\it $j$-reliability ideal} of $S$. The unique minimal monomial generating set of $I_{S,j}$, denoted $G(I_{S,j})$, is formed by the monomials corresponding to the elements of $\overline{\Fc}_{S,j}$ (see \cite[\S 2]{SW09} for more details). Hence, obtaining the set of minimal cuts of $S$ amounts to compute the minimal generating set of $I_{S,j}$.  

\smallskip

In order to compute the {\it $j$-reliability} of $S$ (i.e. the probability that the system is performing at least at level $j$) we can use the numerator of the Hilbert series of $I_{S,j}$, denoted by $H_{I_{S,j}}$. The polynomial $H_{I_{S,j}}$ gives a formula, in terms of $x_1,\dots,x_n$ that enumerates all the monomials in $I_{S,j}$, i.e. the monomials corresponding to the states in $\Fc_{S,j}$. Hence, computing the (numerator of) the Hilbert series of $I_{S,j}$ provides a way to compute the $j$-reliability of $S$ by substituting $x_i^a$ by $p_{i,a}$, the probability that the component $i$ is at least performing at level $a$, as explored in \cite[\S 2]{SW09} (for the binary case).

\smallskip

Often in practice it is more useful to have {\it bounds} on the $j$-reliability of $S$ rather than the exact formula. In order to have a formula that can be truncated at different summands to obtain bounds for the $j$-reliability in the same way that we truncate the inclusion-exclusion formula to obtain the so-called Bonferroni bounds, we need a special way to write the numerator of the Hilbert series of $I_{S,j}$. This convenient form is given by the alternating sum of the ranks in any free resolution of the ideal $I_{S,j}$. Every monomial ideal $I$ has a {\em minimal} free resolution, which provides the tightest bounds among the aforementioned ones. The ranks of the free modules in the minimal free resolution are called the {\em Betti numbers} of the ideal and are denoted by $\beta_i(I)$, or by $\beta_{i,j}(I)$ in the graded case. In general, the closer the resolution is to the minimal one, the tighter the bounds obtained, see e.g. \cite[\S 3]{SW09}.

In summary, the algebraic method for computing the $j$-reliability of a coherent system $S$ works as follows: 
\begin{enumerate}
\item{Associate to the system $S$ its $j$-reliability ideal $I_{S,j}$}.
\item{Obtain the minimal generating set of $I_{S,j}$ to get the set $\overline{\Fc}_{S,j}$.}
\item{Compute the Hilbert series of $I_{S,j}$ to have the $j$-reliability of $S$.}
\item[(3')]{Compute any free resolution of $I_{S,j}$. The alternating sum of the ranks of this resolution gives a formula for the Hilbert series of $I_{S,j}$ i.e., the unreliability of $S$, which provides bounds by truncation at each summand.}
\end{enumerate}

The choice between steps (3) or (3') depends on our needs. If we are only interested in computing the full reliability formula, then we can use any algorithm that computes Hilbert series in step (3). However, if we need bounds for our system reliability, then 
we can compute any free resolution of $I_{S,j}$ and thus perform step (3'). If the performing probabilities of the different components are independent and identically distributed (i.i.d), then in points (3) and (3') of this procedure we only need the graded version of  Hilbert series and free resolutions. Otherwise, we need their multigraded version. For more details and the proofs of the results described here, we refer to \cite{SW09,SW12}. To see more applications of this method in reliability analysis we refer to \cite{SW10,SW11,SW15}.

\section{Simple multi-state $k$-out-of-$n$ systems} \label{Sec:FirstMS_KN}
The first definition of multi-state $k$-out-of-$n$ systems was given by El-Neweihi et al. in the seminal work \cite{EPS78}. They define multi-state systems as follows:
\begin{Definition}[El-Neweihi et al., $1978$]\label{MS-EN}
	A system of $n$ components is said to be a \emph{multi-state coherent system (MCS)} if its structure function $\phi$ satisfies:
	\begin{enumerate}
		\item $\phi$ is increasing.
		\item For level $j$ of component $i$, there exists a vector $(\cdot_i,\xb)$ such that $\phi(j_i, \xb)=j$ while $\phi(l_i,\xb)\neq j$ for $l\neq j$ for $i=1,\dots,n$ and $j=0,\dots,M$.
		\item $\phi(\textbf{j})=j$ for $j=0,\dots,M$, where $\textbf{j}=(j,\dots,j)$.
	\end{enumerate}
\end{Definition}
Where $(j_i, \xb)$ means that the state of the $i$'th component in $\xb$ is $j$. Observe that this definition is more restrictive than ours in the sense that they assume every component has the same number of states, which is in turn the number of states of the system, i.e. $M$. 

The definiton of multi-state $k$-out-of-$n$ systems in \cite{EPS78} is:

\begin{Definition}[El-Neweihi, $1978$] \label{Def:KN_EN_78}
	A system is a \emph{multi-state $k$-out-of-$n$ system} if its structure function satisfies
	\begin{equation}\label{eq:KN_EN_78}
	 \phi(\xb)=x_{(n-k+1)}
	 \end{equation}
	 where $x_{(1)}\leq x_{(2)}\leq\dots\leq x_{(n)}$ is a non decreasing arrangement of $x_1,\dots, x_n$. 
\end{Definition}

Observe that this definition satisfies the conditions given in Definition \ref{MS-EN}. It is easy to check that $\phi$ is an increasing function and $\phi(\textbf{j})=j$ for all $j=0,\dots,M$. To see condition $(2)$ just observe that there always exists a non decreasing arrangement of $x_1,\dots, x_n$ in which  $\phi(j_i, \xb)=j$ while $\phi(l_i,\xb)\neq j$ for $l\neq j$ for $i=1,\dots,n$ and $j=0,\dots,M$. Taking the vector in which the first $n-k+1$ components are lower than $j$ and the rest of the are greater than $j$, we have that condition $(2)$ is satisfied.

\begin{Remark}
	This kind of systems are called \emph{simple multi-state $k$-out-of-$n$ systems} in \cite{KZ03}.
\end{Remark}

We describe now the $j$-reliability ideal of these multi-state $k$-out-of-$n$ systems: 

\begin{Proposition}\label{Prop:ReliabilityKN_EN_BK}
	The ideal \[I_{(k,n),j}=\langle \prod_{\substack{\sigma\subseteq\lbrace1,\dots,n\rbrace \\ \vert\sigma\vert=k}} x_i^j\ \vert\, i\in\sigma \rangle\] is the $j$-reliability ideal of a multi-state $k$-out-of-$n$ system as defined in Definition \ref{Def:KN_EN_78}.
\end{Proposition}

\begin{proof}
	First of all we need to check that all $\mu\in G(I_{(k,n),j})$ satisfy $\phi(\mu)=j$. Let $x^\mu=x_{i_1}^j x_{i_2}^j\dots x_{i_k}^j$ be a generator of $I_{(k,n),j}$, with $\lbrace i_{1},\dots, i_k\rbrace\subseteq \lbrace 1,\dots,n\rbrace$. If we make a non decreasing arrangement of $x_{i_1},\dots, x_{i_k}$ we obtain the vector $(0,...,0,j,...,j)$ in which the first $n-k$ components are in state $0$ and the other components are in state $j$. Applying the structure function $\phi$ to this vector we have that $\phi(0,...,0,j,...,j)=j$.
	
	Now, if $x^\nu\in I_{(k,n),j}$, there exists $x^\mu\in G(I_{(k,n),j})$ such that $\mu\ \leq \nu$. This implies $\phi(\mu)\leq \phi(\nu)$ and since $\phi(\mu)=j$ and $\phi$ is an increasing function, we obtain $\phi(\nu)\geq j$.
	
	Finally if $l<j$ and $\phi(\nu)=l$ we must have $x^\nu\not\in I_{(k,n),j}$. Since $\phi(\nu)=l<j$ we have that there are at most, $k-1$ variables with exponent greater or equal $j$. This implies that there does not exist any $\sigma\in\lbrace 1,\dots,n\rbrace$ with $\vert\sigma\vert=k$ such that $\prod_{x_i\in\sigma} x_i^j\ \vert\ x^\nu$, hence $x^\nu\notin I_{(k,n),j}$. 
\end{proof}

In \cite{BK94} Boedigheimer and Kapur define customer-driven reliability models for multi-state systems. They consider systems with $M$ states in which component $i$ can be in $M_i$ states. They describe such systems using {\em upper and lower boundary points}, which are enough to describe the system completely and are defined as follows 

\begin{Definition}
	We say $\xb$ is a \textit{lower boundary point {\rm (l.b.p.)} to level $j$} iff $\phi(\xb)\geq j$ and $\yb<\xb$ implies that $\phi(\yb)<j$, for $j=1,\dots,M$.
	An \textit{upper boundary point {\rm (u.b.p)} to level $j$} is an $n$-tuple $\xb$ such that $\phi(\xb)\leq j$ and $\yb>\xb$ implies that $\phi(\yb)>j$, for $j=0,\dots,M-1$.
\end{Definition}

Observe that the lower boundary points to level $j$ are the minimal monomial generators of the $j$-reliability ideal of the system. To describe upper boundary points algebraically we need the concept of maximal standard pairs \cite{STV95}.

\begin{Definition}
Let $I$ a monomial ideal in $R=\kb[x_1,\dots,x_n]$ and $\sigma\subseteq\{1,\dots,n\}$. The pair $(x^\mu,\sigma)$ is a {\em standard pair} for $I$ if it satisfies:
\begin{itemize}
\item[-] $\supp(x^\mu)\cap\sigma=\emptyset$, where $\supp(x^\mu)$ is the set of indices $i\in\{1,\dots,n\}$ such that $x_i$ divides $x^\mu$.
\item[-] for all monomials $x^\nu$ such that $\supp(x^\nu)\subseteq\sigma$ we have that $x^\mu x^\nu\notin I$.
\item[-] $(x^\mu,\sigma)\not\subseteq(x^\nu,\tau)$ for any other $(x^\nu,\tau)$ satisfying the two previous conditions.
\end{itemize}
We say that $(x^\mu,\sigma)$ is a {\em maximal standard pair} if there is no other standard pair $(x^\nu,\sigma)$ such that $x^\mu$ divides $x^\nu$.
\end{Definition}

Maximal standard pairs are in one-to-one correspondence with upper boundary points.

\begin{Theorem}
Let $I_{S,j}$ be the $j$-reliability ideal of a coherent system $S$ for which component $i$ can be in states  $(0,\dots,M_i)$. Then $\mu+\sum_{i\in\sigma} 1_{M_i}$ is an upper boundary point of $S$ for level $j-1$ if and only if $(x^\mu,\sigma)$ is a maximal standard pair of $I_{S,j}$.
\end{Theorem}
\begin{proof}

\noindent{$\implies$)}
Let $\alpha$ be an upper boundary point of $S$ for level $j-1$. Let $\sigma\subseteq \{1,\dots,n\}$ be the set of components of $S$ such that $\alpha_i=M_i$. We have that $\sigma\neq\{1,\dots,n\}$ i.e. there exists at least one component $i$ such that $\alpha_i\neq M_i$ hence $\alpha$ is of the form $\alpha=\mu+\sum_{i\in\sigma}1_{M_i}$. $\phi(\alpha)<j$ implies $x^\alpha\notin I_{S,j}$, and we claim that $(\mu,\sigma)$ is a standard pair for $I_{S,j}$. To see this, let $x^\mu x^\nu$ such that $\supp(x^\nu)\subseteq \sigma$. If $\nu_i\leq M_i$ then clearly $x^\mu x^\nu\notin I_{S,j}$ because $\mu+\nu\leq\alpha$ and $\phi(\alpha)<j$. Now, since $x^\alpha\notin I_{S,j}$ we know there is no minimal generator of $I_{S,j}$ that divides $x^\alpha$ and since $M_i=\alpha_i$ is the maximal power to which variable $i$ can possibly be raised to in any generator of $I_{S,j}$ then no generator will divide $x^\alpha x^\nu$ for any $\nu$ such that $\supp(x^\nu)\subseteq \sigma$ hence $(\mu,\sigma)$ is a standard pair. Assume now that $(\mu,\sigma)$ is not maximal. Then there is some $i'\notin\sigma$ such that $(\mu+1_{i'},\sigma)$ is a standard pair for $I_{S,j}$. Then $x^\mu x_{i'}\prod_{i\in\sigma}x_i^{M_i}\notin I_{S,j}$  i.e. $\phi(\alpha+1_i)<j$ which contradicts the assumption that $\alpha$ is an upper boundary point of $S$ for level $j-1$.

\noindent{$\Leftarrow$)}
Let $(x^\mu,\sigma)$ be a maximal standard monomial of $I_{S,j}$, i.e. $x^\mu\notin I_{S,j}$ and $x^\mu x^\nu\notin I_{S,j}$ for all $x^\nu$ such that $\supp(x^\nu)\subseteq\sigma$. Let $x^\alpha=x^\mu\prod_{i\in\sigma}x_i^{M_i}$. Since $x^\alpha\notin I_{S,j}$ we know that $\phi(\alpha)<j$.
Let now $\beta>\alpha$, we can assume without loss of generality that $\beta=\alpha+1_i$ for some $i\notin \sigma$. Suppose $x^\beta\notin I_{S,j}$. Then there is no minimal generator of $I_{S,j}$ that divides $x^\beta$ but since $M_i$ is the maximal state of component $i$, then there is no minimal generator of $I_{S,j}$ that divides $x^\beta x^\nu$ for any $\nu$ such that its support is a subset of $\sigma$. Finally since the difference between $x^\mu x_i$ and $x^\beta$ is a monomial whose support is in $\sigma$, we have that $(x^\mu x_i,\sigma)$ is a standard pair for $I_{S,j}$, which is in contradiction with the fact that $(x^\mu,\sigma)$ is maximal, hence $x^\beta\in I_{S,j}$ and $\alpha$ is an upper boundary point of $S$ for level $j-1$.
\end{proof}

Using upper and lower boundary points, Boedigheimer and Kapur define multi-state $k$-out-of-$n$ systems as follows.

\begin{Definition}[Boedigheimer and Kapur, $1994$]\label{Def:KN_BK_94}
	$\phi$ is a multi-state $k$-out-of-$n:G$ structure function if, and only if, $\phi$ has ${n \choose k}$ lower boundary points to level $j$ $(j=1,\dots, M)$ and ${n \choose k-1}$ upper boundary points to level $j$ $(j=0,\dots, M-1)$.
\end{Definition}

The minimal generating set of the ideal $I_{(k,n),j}$ in Proposition \ref{Prop:ReliabilityKN_EN_BK} has ${n\choose k}$ elements, i.e. this system has ${n\choose k}$ lower boundary points. The maximal standard pairs of $I_{(n,k),j}$ are $(\prod_{i\in\sigma} x_i^{j-1},\{1,\dots,n\}-\sigma)$ for all $\sigma\subseteq\{1,\dots,n\}$ such that $\vert\sigma\vert=n-k+1$, i.e. the number of upper boundary points of $S$ for $j-1$ is ${n\choose {n-k+1}}={n\choose {k-1}}$. Hence, Proposition \ref{Prop:ReliabilityKN_EN_BK} is a proof of the equivalence of definitions \ref{Def:KN_EN_78} and \ref{Def:KN_BK_94} in the case that $M_i=M$ for all $i$.

If we allow that the number of states of each of the components can be different, then the situation is more complicated. Let $n_j$ be the number of components such that their maximum performance level $M_i$ is bigger than or equal to $j$. If $n_j\geq k$ then the system behaves as a multi-state $k$-out-of-$n$ system by setting $\phi$ as in Definition \ref{Def:KN_EN_78}. The number of lower and upper boundary points does however vary. The lower boundary points are given by the tuples that have $k$ components at level $j$ and $n-k$ components at level $0$, and there are ${{n_j}\choose k}$ such tuples. And if $n_j\geq k$ then the upper boundary points for level $j$ are given by the tuples in which $k-1$ components are at their maximum level (strictly bigger than $j$), the other component such that its maximum level is bigger than $j$ is exactly at level $j$ and the rest of the components are at level $\min\{M_i,j\}$. The number of such tuples is ${n_{j+1}\choose k}$. Hence the system behaves at level $j$ as a $k$-out-of-$n_j$ system according to definition \ref{Def:KN_BK_94}. In fact, if we only consider those components whose maximum performance level is bigger than $j$ then the system behaves at level $j$ as a $k$-out-of-$n_j$ system according to both definitions.

We can then generalize the ideal in Proposition \ref{Prop:ReliabilityKN_EN_BK} allowing different number of levels for each component:
\begin{Definition}\label{def:K_N_alg}
Let $S$ be a multi-state system with levels $\{0,\dots,M\}$ and such that each component $i$ has $M_{i+1}$ levels of performance $\{0,\dots, M_i\}$. Let $n_j\leq n$ be the number of components such that $M_i\geq j$ for each $j\in\{0,\dots,M\}$ (for ease of notation we consider that these are components $1,\dots,n_j$).
 $S$ is a multi-state $k$-out-of-$n$ system if for every $j\in\{1,\dots,M\}$  the $j$-reliability ideal of $S$, $I_{S,j}$, is of the form
 \[
 I_{S,j}=\langle \prod_{\substack{\sigma\subseteq\lbrace 1,\dots, n_j\rbrace \\ \vert\sigma\vert=k}} x_i^j\ \vert\, i\in\sigma \rangle.
 \] 
\end{Definition}

\begin{Example}\label{Ex:KN_BK}
	Let $S$ be a system such that $\mathcal{S}_1=\{0,1,2,3,4\},\, \mathcal{S}_2=\{ 0,1,2,3\},\, \mathcal{S}_3=\mathcal{S}_4=\{0,1,2\}$ and $\mathcal{S}_5=\{ 0,1\}$ and let $\phi(\xb)=x_{(4)}$. Observe that $n_1=5$, $n_2=4$, $n_3=2$, $n_4=1$. The system behaves as a $2$-out-of-$5$ for levels $j=1,2,3$ according to Definition \ref{Def:KN_EN_78} and as a $2$-out-of-$n_j$ system for levels $j=1,2,3$ according to Definition \ref{Def:KN_BK_94}. The lower and upper boundary points are given in Table \ref{table:exampleKN}.
	
\begin{center}
\begin{small}{ \begin{table}
   \begin{tabular}{cll}
	Level&Lower boundary points &Upper boundary points \cr
	\hline
	$0$& &$(0,0,0,0,1),(0,0,0,2,0),(0,0,2,0,0),$ \cr
	& &$(0,3,0,0,0),(4,0,0,0,0)$ \cr
	$1$&$(0,0,0,1,1),(0,0,1,0,1),(0,1,0,0,1),(1,0,0,0,1),$&$(1,1,1,2,1),(1,1,2,1,1),(1,3,1,1,1),$ \cr
	&$(0,0,1,1,0),(0,1,0,1,0),(1,0,0,1,0),(0,1,1,0,0)$&$(4,1,1,1,1)$\cr
	&$(1,0,1,0,0),(1,1,0,0,0)$&\cr
	$2$&$(0,0,2,2,0),(0,2,0,2,0),(2,0,0,2,0),(0,2,2,0,0),$&$(2,3,2,2,1),(4,2,2,2,1)$\cr
	&$(2,0,2,0,0),(2,2,0,0,0)$& \cr
	$3$&$(3,3,0,0,0)$& \cr
	\hline
  \end{tabular}\caption{Upper and lower boundary points for the system in Example \ref{Ex:KN_BK}}\label{table:exampleKN}		
 \end{table}
 }\end{small}
\end{center}
The reliability ideals for this system are
\begin{align*}
I_{S,1}&=\langle x_1x_2,x_1x_3,x_1x_4,x_1x_5,x_2x_3,x_2x_4,x_2x_5,x_3x_4,x_3x_5,x_4x_5\rangle\\
I_{S,2}&=\langle x_1^2x_2^2,x_1^2x_3^2,x_1^2x_4^2,x_2^2x_3^2,x_2^2x_4^2,x_3^2x_4^2\rangle\\
I_{S,3}&=\langle x_1^3x_2^3\rangle.
\end{align*}

\end{Example}

%
%
%
%

\section{Generalized multi-state $k$-out-of-$n$ systems}\label{Sec:GeneralizedKN}
In \cite{HZW00} Huang, Zuo and Wu introduced {\em generalized multi-state $k$-out-of-$n$ systems} allowing different number of components for a system to perform at each level $j$ naturally extending the capabilities of the systems studied in the previous section and providing more flexibility to describe practical situations. The definition in \cite{HZW00} is the following

\begin{Definition}[Huang, Zuo and Wu, $2000$]\label{Def:KN_HZW_00}
	An $n$-component system is called a \emph{generalized multi-state $k$-out-of-$n$:G system} if $\phi(\xb)>j,\ 1\leq j\leq M$ whenever there exists an integer value $l$ $(j\leq l\leq M)$ such that at least $k_l$ components are in state $l$ or above.
\end{Definition}

If we denote by $\phi$ the structure function of the system $S$ and by $N_j$ the number of components in state $j$ or above, then this definition can be rephrased by saying that $\phi(S)\geq j$ if

\begin{align*} 
N_j &\geq  k_j \\ 
N_{j+1} &\geq  k_{j+1} \\
  \,&\,\vdots \\
N_M &\geq k_M 
\end{align*}

Hence we can denote a generalized multi-state $k$-out-of-$n$ system by $S_{n,(k_1,\dots,k_M)}$. When $k_1\leq\cdots\leq k_m$ the system is called an {\em increasing} generalized multi-state $k$-out-of-$n$:G system, and if $k_1\geq\cdots \geq k_m$ the system is said to be {\em decreasing}. Huang {\em et al.} provide formulas for both cases and an enumerative algorithm for the evaluation of the reliability of generalized multi-state $k$-out-of-$n$ systems when the sequence $(k_1,\dots,k_M)$ is monotone.

Continuing this line M. J. Zuo and Z. Tian defined in \cite{ZT06} generalized multi-state $k$-out-of-$n$:F systems.

\begin{Definition}[Zuo and Tian, $2006$]\label{Def:KN_ZT_06}
An $n$-component system is called \emph{generalized multi-state $k$-out-of-$n:F$ system} if $\phi(\xb)<j,\ 1\leq j\leq M$ whenever the states of at least $k_l$ components are below $l$ for all $l$ such that $j\leq l\leq M$.
\end{Definition}

Using this definition they provide a correspondence between generalized multi-state $k$-out-of-$n$:G systems and generalized multi-state $k$-out-of-$n$:F systems. They study these systems when the sequence $(k_1,\dots, k_M)$ is not necessarily monotone and provide an efficient algorithm that is recursive on $M$, the number of performance levels. This algorithm outperforms the one in \cite{HZW00} which is recursive in $n$.

Using the ideals in Proposition \ref{Prop:ReliabilityKN_EN_BK} we can immediately describe the reliability ideal of a generalized multi-state $k$-out-of-$n$:G system given by $(k_1,\dots,k_M)$.

\begin{Proposition}\label{Prop:genMSKN}
The $j$-reliability ideal of a generalized multi-state $k$-out-of-$n$ system $S=S_{n,(k_1,\dots,k_M)}$ is given by
\[
I_{S,j}=I_{n,(k_j,\dots,k_M)}=\sum_{i=j}^{M}I_{(k_i,n),i}.
\]
\end{Proposition}

\begin{Example}\label{ex:alg_gen_multi-state}
We study here Example 8 in \cite{HZW00} with the algebraic method and recover the exact same results given there. The system in this example is a generalized multi-state $k$-out-of-$3$:G system with four states $(0,1,2,3)$ such that $k_1=3$, $k_2=2$ and $k_3=2$, hence it is a {\em decreasing} generalized multi-state $k$-out-of-$n$:G system. The probabilities of the different components are given by $p_{1,0}=0.1$, $p_{1,1}=0.2$, $p_{1,2}=0.3$, $p_{1,3}=0.4$, $p_{2,0}=0.1$, $p_{2,1}=0.1$, $p_{2,2}=0.2$, $p_{2,3}=0.6$, $p_{3,0}=0.1$, $p_{3,1}=0.2$, $p_{3,2}=0.4$, $p_{3,3}=0.3$, where $p_{i,j}$ is the probability that component $i$ is performing at level $j$.
\begin{itemize}
\item[-] For the system to be in state $3$ there must be at least  $2$ components in state $3$ or above ($k_3=2$). Hence the corresponding ideal is $I_{S,3}=\langle x^3y^3,x^3z^3,y^3z^3\rangle$. The numerator of the Hilbert series is $H_{I_{S,3}}=x^3y^3+x^3z^3+y^3z^3-2(x^3y^3z^3)$ and when plugging the probabilities in, we have that the probability that the system is in state $3$ or above, denoted $R_{S,3}$, is $0.396$, which equals the probability that the system is exactly in state $3$, denoted $r_{S,3}$.
\item[-] The system is in state $2$ or above if at least $2$ components are in state $2$ or above, hence $I_{S,2}=I_{(2,3),2}+I_{(2,3),3}=I_{(2,3),2}=\langle x^2y^2,x^2z^2,y^2z^2\rangle$. The numerator of the Hilbert series is $H_{I_{S,2}}=x^2y^2+x^2z^2+y^2z^2-2(x^2y^2z^2)$ and we obtain $R_{S,2}=0.826$ and $r_{S,2}=R_{S,2}-R_{S,3}=0.826-0.396=0.430$.
\item [-] Since $k_1=3$ the system is in state $1$ or above if all $3$ components are in state $1$ or above or if at least $2$ components are in state $2$ or above or if at least $2$ components are in state $3$ or above. The corresponding ideal is then $I_{S,1}=I_{(3,3),1}+I_{(2,3),2}+I_{(2,3),3}=I_{(3,3),1}+I_{(2,3),2}=\langle xyz,x^2y^2,x^2z^2,y^2z^2\rangle$, $H_{I_{S,1}}=xyz+x^2y^2+x^2z^2+y^2z^2-(xy^2z^2+x^2yz^2+x^2y^2z)$ and we obtain $R_{S,1}=0.89$ and $r_{S,1}=R_{S,1}-R_{S,2}=0.89-0.826=0.064$.
\item[-] Finally $r_{S,0}=R_{S,0}-R_{S,1}=1-0.89=0.11$.
\end{itemize}
\end{Example}

Using the reliability ideals of generalized multi-state $k$-out-of-$n$:G systems given in Proposition \ref{Prop:genMSKN} we can develop a recursive method to compute their reliability. The method is recursive on $M$, the number of performance levels and can be used for any sequence $(k_1,\dots,k_M)$ describing the system, not necessarily monotone. This method is an enumerative one that can be used even when the component's probabilities are not i.i.d. For the i.i.d. case our method is equivalent to the one in \cite{ZT06} in terms of computational complexity. We will use the technique of Mayer-Vietoris trees, which were introduced in \cite{S08,S09}, see Appendix \ref{appendix:MVT} for an explanation of the method. For ease of the notation we assume that the sequence $(k_1,\dots,k_M)$ is strictly decreasing. In any other case, the only difference is that some of the summands that compose the ideal $I_{n,(k_j,\dots,k_M)}$ will be missing, as we saw in Example \ref{ex:alg_gen_multi-state} but this fact does not affect the algorithm description or its performance.

Let $1\leq j\leq M$ and $I_{n,(k_j,\dots,k_M)}=\sum_{i=j}^{M}I_{(k_i,n),i}$ the $j$-reliability ideal of the system. We sort the generators of $I_{n,(k_j,\dots,k_M)}$ in ascending degree and lexicographically within each degree. For constructing the Mayer-Vietoris tree we will use as pivot always the last generator. First, we use as pivots the generators of $I_{(k_M,n),M}$. We denote each of them by $x_\sigma^M=\prod_{x_i\in\sigma}x_i^M$ for $\sigma\subseteq \{1,\dots, n\}$ and $\vert\sigma\vert=k_M$. For each of these generators we obtain as left child in the Mayer-Vietoris tree the ideal denoted by $I_{\sigma,M}$ given by

$$I_{\sigma,M}=I_{n-k_M,(k_j-k_M,\dots,k_{M-1}-k_M)}+\sum_{x_i\notin \sigma, x_i<\max(\sigma)}\langle x_i^M \rangle,$$

where $I_{n-k_M,(k_j-k_M,\dots,k_{M-1}-k_M)}\subseteq \kb[[n]-\sigma]$. On each of the nodes of the tree we use as pivots the monomials in $\sum_{x_i\notin \sigma, x_i<\max(\sigma)}\langle x_i^M \rangle$ and proceed in the same way when the node is $I_{\sigma,M}=I_{n-k_M,(k_j-k_M,\dots,k_{M-1}-k_M)}$. Finally, after using all the generators of $I_{n,(k_j,\dots,k_M)}$ as pivots, we are left with the ideal $I_{n,(k_j,\dots,k_{M-1})}$. This procedure leads to the following recursive formula for the Betti number of $I_{n,(k_j,\dots,k_M)}$ (we give here the version for i.i.d. components)

\begin{align}\label{ms-k-out-of-n-formula}
\begin{split}
\beta_\alpha(I_{n,(k_j,\dots,k_M)})&=\beta_\alpha(I_{n,(k_j,\dots,k_{M-1})})\\
&+\sum_{i=0}^{n-k_M-2}{{n}\choose{k_M+i}}{{i+k_M-1}\choose{k_M-1}}p_{\geq M}^{k_M+i}\beta_{\alpha-i+1}(I_{n-k_M-i,(k_j-k_M-i,\dots,k_{M-1}-k_M-i)})\\
&+{{n}\choose{k_M+\alpha-1}}{{\alpha+k_M-2}\choose{k_M-1}}p_{\geq M}^{k_M+\alpha-1}\left(\sum_{i=j}^{M-1}{{n-k_M-(\alpha-1)}\choose{k_i-k_M-(\alpha-1)}}p_{\geq i}^{k_i-k_M-(\alpha-1)}\right)\\
&+ p_{\geq M}^{k_M+\alpha} \sum_{i=1}^{n-k_M}(i+1){\i \choose \alpha}.\\
\end{split}
\end{align}

The complete derivation of this formula is straightforward but somewhat tedious. It is based on the analysis of the branches of the Mayer-Vietoris tree, as described in Appendix \ref{appendix:MVT}.
Observe that the computation for $(k_1,\dots,k_M)$ is done in terms of cases with strictly less than $M$ levels, and hence the recursion is on the number of performance levels, and not on the number of variables. The efficiency of this method is equivalent to the one in \cite{ZT06}.

\begin{Remark} There are several algorithms to compute the reliability of generalized multi-state $k$-out-of-$n$ systems. Some of them are restricted to identical independent components. Among these, the algorithm in \cite{HZW00} is as we have seen enumerative (hence of low efficiency) and applicable to monotonic patterns, the one in \cite{ZT06} is also enumerative but more efficient and is applicable to monotonic and non-monotonic patterns. The algorithm in \cite{CBAZ12} is non enumerative and more efficient than the previous ones. For the case of independent but not necessarily identical components the algorithm by \cite{ZC10} uses a finite Markov chain imbedding (FMCI) approach and is adequate for small size systems, as is the algorithm in \cite{TZY08}. Other more efficient algorithms include \cite{CBAZ12}, based on conditional probabilities, or  \cite{MXAB15} using multi-valued decision diagrams. Our algebraic approach is enumerative and applicable to both kind of systems (with independent and identical components and with independent non identical components) and produces not only the full reliability formulas but also bounds.
\end{Remark}

\subsection{Quality of the algebraic bounds}
For a polynomial ring $R=\kb[x_1,\dots,x_n]$ Hilbert's syzygy theorem (cf. \cite{E95} for instance) states that the length of any resolution of an ideal in $R$ is bounded above by $n+1$. In our context this means that the algebraic method using the Betti numbers of reliability ideals produces a compact version of the inclusion-exclusion identity and thus a series of Bonferroni-like bounds for the system's reliability such that if the system $S$ has $n$ components then the reliability formula, given by the Hilbert series numerator of $I_S$, has at most $n+1$ summands. Every truncation of this formula provides a bound for the reliability. We compare these bounds with the following ones considered in \cite{GN17} for some generalized $k$-out-of-$n$ multi-state systems.

If we denote by $\yb^m,\, m=1,\dots, M_p$ the minimal path vectors of  a given multi-state system $S$,with structure function $\phi$ then a lower bound for the reliability of $S$ is given (assuming independent components) by
\[
l'_{\phi}(\pb)=\max_{1\leq m\leq M_p}(\prod_{i=1}^n P(x_i\geq y_i^m))=\max_{1\leq m\leq M_p}(\prod_{i=1}^n p_i^{y_i^m}).
\]
On the other hand, if the minimal cuts of $S$ are given by $\zb^m,\, m=1,\dots,M_c$ then we have the lower bound
\[
l^{**}_{\phi}(\pb)=\prod_{m=1}^{M_c}\coprod_{i=1}^{n}P(x_i\geq z_i^m))=\prod_{m=1}^{M_c}\coprod_{i=1}^{n}p_i^{z_i^m+1}
\]
where for real numbers $p\in[0,1]$ we define $\coprod_{i=1}^{n}p_i=1-\prod_{i=1}^{n}(1-p_i)$.

\begin{Example}
Let $k_1=4,k_2=2,k_3=1$ and let $n=8,11,14$. Let us consider the multi-state generalized $k$-out-of-$n$:G systems $I_{n,(4,2,1)}$ for the following probabilities, independent but not identical:

\begin{small}
\begin{center}
\begin{table}[h]
\begin{tabular}{|r|c|c|c|c|c|c|c|c|c|c|c|c|c|c|}
\hline
level&$c_1$&$c_2$&$c_3$&$c_4$&$c_5$&$c_6$&$c_7$&$c_8$&$c_9$&$c_{10}$&$c_{11}$&$c_{12}$&$c_{13}$&$c_{14}$\\
\hline
1&0.5&0.6&0.4&0.5&0.6&0.4&0.5&0.6&0.4&0.5&0.6&0.4&0.5&0.6\\
2&0.2&0.15&0.1&0.2&0.15&0.1&0.2&0.15&0.1&0.2&0.15&0.1&0.2&0.15\\
3&0.1&0.05&0.05&0.1&0.05&0.05&0.1&0.05&0.05&0.1&0.05&0.05&0.1&0.05\\
\hline
\end{tabular}
 \caption{Probabilities $p_{i,j}$, i.e. $P(c_i\geq j)$ for the components of several generalized multistate $k$-out-of-$n$ systems} 
 \end{table}
\end{center}
\end{small}

The number of generators (i.e. number of minimal paths) of each of the systems considered are given in Table \ref{table:generators} we also give the number of minimal cuts.

\begin{small}
\begin{center}
\begin{table}[h]
\begin{tabular}{|l|c|c|c|}
\hline
Sytem&level&$\#$ minimal paths&$\#$ minimal cuts\\
\hline
$S_{8,(4,2,1)}$&1&106&168\\
$S_{8,(4,2,1)}$&2&36&8\\
$S_{8,(4,2,1)}$&3&8&1\\
$S_{11,(4,2,1)}$&1&396&495\\
$S_{11,(4,2,1)}$&2&66&11\\
$S_{11,(4,2,1)}$&3&11&1\\
$S_{14,(4,2,1)}$&1&1106&1092\\
$S_{14,(4,2,1)}$&2&105&14\\
$S_{14,(4,2,1)}$&3&14&1\\
\hline
\end{tabular}
 \caption{Number of minimal paths and cuts for several generalized multistate $k$-out-of-$n$ systems} 
 \label{table:generators}
 \end{table}
\end{center}
\end{small}

The results are summarized in tables \ref{table:experiments1l} and \ref{table:experiments1u} in which we consider the probability of the system performing at levels $1$ to $3$. In the tables, column $l_i$ indicates a lower bound given by the first $i$ summands of the Hilbert series numerator of the corresponding $j$-reliability ideal, while column $u_i$ denotes an upper bound given by the first $i$ summands. An asterisk indicates that the bound is sharp. Cells with a minus sign $-$ indicate that the bound is meaningless (i.e. upper bounds above $1$ or lower bounds below $0$).

\begin{small}
\begin{table}
\begin{tabular}{|l|c|c|c|c|c|c|c|c|c|c|c|c|c|}
\hline
System&Level&$l_2$&$l_4$&$l_6$&$l_8$&$l_{10}$&$l_{12}$&$l_{14}$\\
\hline
$S_{8,(4,2,1)}$&1& - & - &0.419984&0.779916& & &\\
$S_{8,(4,2,1)}$&2& - &0.480262&0.530988&0.531611& & &\\
$S_{8,(4,2,1)}$&3&0.42&0.435844&0.435914*& & & &\\
\hline
$S_{11,(4,2,1)}$&1& - & - & - & - &0.0.914949&0.937376*&\\
$S_{11,(4,2,1)}$&2& - &0,357057&0.654349&0.666748&0.666865&0.666866*&\\
$S_{11,(4,2,1)}$&3&0.4975&0.541256&0.541819&0.541821*&& &\\
\hline
$S_{14,(4,2,1)}$&1&&& - & - & - &0.870386&0.984878\\
$S_{14,(4,2,1)}$&2&&&0.670885&0.765189&0.767655&0.767675*& \\
$S_{14,(4,2,1)}$&3&&&0.627826&0.627844*&  & & \\
\hline
\end{tabular}
\caption{Lower bounds for several generalized multi-state $k$-out-of-$n$ systems.}
\label{table:experiments1l}
\end{table}

\begin{table}
\begin{tabular}{|l|l|c|c|c|c|c|c|c|c|}
\hline
System&Lvl.&$u_1$&$u_3$&$u_5$&$u_7$&$u_9$&$u_{11}$&$u_{13}$&$u_{15}$\\
\hline
$S_{8,(4,2,1)}$&1& - & - & - &0.825892&0.782246*& & & \\
$S_{8,(4,2,1)}$&2& - &0.750481&0.538913&0.531642&0.531612*& & &\\
$S_{8,(4,2,1)}$&3&0.55&0.43725&0.435916&0.435914* & & & &\\
\hline
$S_{11,(4,2,1)}$&1& - & - & - & - & - &0.938269&0.937376* &\\
$S_{11,(4,2,1)}$&2& - &-&0.741715&0.668326&0.666872&0.666866*& &\\
$S_{11,(4,2,1)}$&3&0.75&0.547875&0.541858&0.541821*&& & &\\
\hline
$S_{14,(4,2,1)}$&1&- & - & -  & -  & - &-&0.992941&0.985126*\\
$S_{14,(4,2,1)}$&1&-&-& - &0.785541&0.767936&0.767677&0.767675*&\\
$S_{14,(4,2,1)}$&1&0.95&0.6455&0.628081&0.627845&0.627844*& & &\\
\hline
\end{tabular}
\caption{Upper bounds for several generalized multi-state $k$-out-of-$n$ systems.}
\label{table:experiments1u}
\end{table}
\end{small}

The results in tables \ref{table:experiments1l} and \ref{table:experiments1u} allow us to discuss the strengths and weaknesses of our method. First of all, for systems with big number of generators, the first bounds are useless due to the fact that each of the first summands of the compact inclusion-exclusion formula consists of a large number of inner summands. As the number of variables increases, we obtain a collection of useful bounds, that compare well with the bounds considered in \cite{GN17} as we can see in Table \ref{table:gnbounds}. Observe that $l^{**}_{\phi}(\pb)$ behaves very well in case we have a multistate parallel system, as is the case in level $3$ of our systems. This is because the minimal cuts are unique in these cases. We have considered low working probabilities in our system, since our bounds are sharper in this case. In case our probabilities are high we can consider the unreliability of the dual systems and thus obtain close bounds. All our bounds were computed in less than one second on a laptop\footnote{CPU: intel i7-4810MQ, 2.80 GHz. RAM: 16Gb}. It is worth noting that the performance of our method does not depend on having identical or non-identical probability distributions in the components of the system.

\begin{table}[h]
\begin{tabular}{|l|l|c|c|}
\hline
System&Lvl.&$l'_{\phi}(\pb)$&$l^{**}_{\phi}(\pb)$\\
\hline
$S_{8,(4,2,1)}$&1&0.108&0.0510583\\
$S_{8,(4,2,1)}$&2&0.1&0.0710738\\
$S_{8,(4,2,1)}$&3&0.1&0.435914*\\
\hline
$S_{11,(4,2,1)}$&1&0.1296&0.35674\\
$S_{11,(4,2,1)}$&2&0.1&0.125414\\
$S_{11,(4,2,1)}$&3&0.1&0.541821*\\
\hline
$S_{11,(4,2,1)}$&1&0.1296&0.762837\\
$S_{11,(4,2,1)}$&2&0.1&0.211015\\
$S_{11,(4,2,1)}$&3&0.1&0.627844*\\
\hline
\end{tabular}
\caption{Lower bounds considered in \cite{GN17} for some generalized multi-state $k$-out-of-$n$ systems}
\label{table:gnbounds}
\end{table}
\end{Example}
\section{Binary $k$-out-of-$n$ system with multi-state components}\label{Sec:B_MS_KN}
The following multi-state generalization of $k$-out-of-$n$ systems was introduced in \cite{S08}. Let $S_{m,n,k}$ be a system with $k$ components, each of which can be in a set of states $\{0,1,\dots, m\}$. $S_{m,n,k}$ is called an $m$-multi-state $k$-out-of-$n$:G system if the system works whenever the sum of the states of the $n$ components is bigger than or equal to $k$. Note that this kind of systems allows $k$ to be bigger than $n$. This is an example of a binary system with multi-state components. This kind of systems are useful to model different situations like the following examples:
\begin{itemize}
\item[-] A storehouse has $n$ storage facilities each of which has a capacity of $m$ units. At any given time each of the facilities is partially full, leaving a real capacity smaller than or equal to $m$ units. The system is said to work if it is capable to store a new arriving lot that consists of $k$ storage units.
\item[-] A set of $n$ pumps and pipes contributes to a global pipe that covers the needs of a power plant. Each individual pipe may supply water at different levels $\{0,\dots,m\}$ and we consider that the system is working if the combined supply (sum of all the individual supplies) is above level $k$.
\end{itemize}
The reliability ideal of $S_{m,n,k}$, denoted by $J^m_{[n,k]}$ is generated by all monomials $x^\mu$ in $n$ variables such that the degree of $x^\mu$ is $k$ and $\mu_j\leq m$ for all $1\leq j\leq n$. To obtain the number of generators of the system (i.e. the minimal working states) and the Betti numbers, needed to compute the reliability function and bounds for it in the algebraic approach, we can proceed as follows.

First, we list all the generators in a precise ordering, following Proposition 3.2.14 in \cite{S08}:  For each $i$ from $m$ descending to $0$ and for each variable $x_j$ for $j$ from $1$ to $n$ (we call $x_j$ the {\em distinguished variable} in each step) we form all monomials $x^\mu$ such that
\begin{itemize}
\item[-] the first $j-1$ variables have an exponent strictly smaller than $i$
\item[-] the variable $x_j$ has an exponent equal to $i$
\item[-] the remaining last $n-j$ variables have an exponent smaller than or equal to $i$
\item[-] the degree of $x^\mu$ equals $k$
\end{itemize}

Using this ordering and Corollary 3.2.25 in \cite{S08} we can obtain the Betti numbers of $J^m_{[n,k]}$ using only one more piece of information, namely, for each generator $x^\mu$ of $J^m_{[n,k]}$ we need to know the number of variables before $x_j$ that have a nonzero exponent in $x^\mu$. So when we list the generators of $J^m_{[n,k]}$ we keep track of how many of the first $j-1$ variables have a nonzero exponent with the notation we just described. The method for this computation of the Betti numbers of a monomial ideal is described in detail in \cite{S08,S09}.

For this, let $j$ be the distinguished variable and $i\leq m$ fixed, the exponent of $x_j$ in $x^\mu$. Now, for each $p$ between $0$ and $k-i$, which represents the sum of the exponents of the first $j-1$ variables of $x^\mu$, and for each $l$ between $0$ and $j-1$, which represents the number of variables among the first $j-1$ ones whose exponent is different from zero, we count all the possible ways to obtain the sum $p$ using $l$ summands each of which is between $1$ and $i-1$. This number is called the {\em number of restricted compositions of $p$ in $l$ summands between $1$ and $i-1$} and is denoted $C(p,l,1,i-1)$ in \cite{JVZ10}. Since we have $l$ nonzero summands among the first $j-1$ variables, we can choose them in ${j-1}\choose {l}$ ways. For each of these choices we have that the exponents of the last $n-j$ variables sum up to $k-i-p$ and each of these exponents is between $0$ and $i$. The number of such compositions is $C(k-i-p, n-j,0,i)$. Hence, putting all these considerations together we have the following result.

\begin{Lemma} \label{lemma:formula}
The number of generators of $J^m_{[n,k]}$ is 
\begin{equation}\label{eq:numgens}
N^m_{[n,k]}=\sum_{i=0}^k \sum_{j=1}^n \sum_{p=0}^{k-i}\sum_{l=0}^{j-1} C(p,l,1,i-1){{j-1}\choose{l}} C(k-i-p,n-j,0,i).
\end{equation}
All these generators have degree $k$, hence $\beta_{0,k}(J^m_{[n,k]})=N^m_{[n,k]}$ and $\beta_{0,j}(J^m_{[n,k]})=0$ for all $j\neq k$. Each generator contributes to $\beta_{i,k+i}(J^m_{[n,k]})$ with ${n-l-1}\choose{i}$ elements, hence the formula for the Betti numbers of $J^m_{[n,k]}$ is
\begin{equation}\label{eq:betti}
\beta_{i,k+i}(J^m_{[n,k]})=\sum_{i=0}^k \sum_{j=1}^n \sum_{p=0}^{k-i}C(p,l,1,i-1){{j-1}\choose{l}} C(k-i-p,n-j,0,i){{n-l-1}\choose{i}}
\end{equation}
and $\beta_{i,j}(J^m_{[n,k]})=0$ if $j\neq k+i$.
\end{Lemma}

\begin{Remark}
The number of restricted compositions of an integer with a given number of bounded summands can be obtained using a certain generating function, as shown in \cite{A76,E13,FS09}. The following closed formula for some types of restricted compositions can be found in Theorem 2.1 in \cite{JVZ10} which can be used to explicitly compute the numbers in Lemma \ref{lemma:formula} using that $C(k-i-p,n-j,0,i)=C(k-i-p+n-j,n-j,1,i+n+j)$:
\[
C(n,k,1,b)=\sum_{\substack{i_2=\alpha_2,i_3,\dots,i_b \\ \max\{0,\alpha_j\}\leq i_j\leq \min\{\beta_j,\gamma_j\}}} \prod_{l=2}^b {{k-\sum_{j=2}^{l-1} i_j}\choose{i_l}},
\]
where

$$
\alpha_j=n-k(j-1)-\sum_{l=j+1}^b (l-j+1)i_l
$$
$$
\beta_j=k-\sum_{l=j+1}^b i_l \\
$$
$$
\gamma_j=\lfloor \frac{n-k-\sum_{l=j+1}^b (l-1)i_l}{j-1}\rfloor .
$$

\end{Remark}

In order to obtain the necessary information to construct the reliability polynomial and bounds from the Betti numbers of $J^m_{[n,k]}$ we need their multigraded version. For this, let $x^\mu$ a minimal generator of $J^m_{[n,k]}$ and $x_j$ its distinguished variable. Let $(x_{i_1},\dots,x_{i_l})$ be the $l$ variables among the first $j-1$ that appear with a nonzero exponent in $x^\mu$. Let $P_{x^\mu}=\{x_1,\dots, \hat{x}_j,\dots, x_n\}\setminus\{x_{i_1},\dots,x_{i_l}\}$. Then the multidegrees of the contribution of $x^\mu$ to $\beta_{i,k+i}(J^m_{[n,k]})$ are $x^\mu\prod_{x_i\in\sigma}x_i$ for each subset $\sigma$ of $P_{x^\mu}$ of cardinality $i$. Observe that the resolution of $J^m_{[n,k]}$ is $k$-linear, i.e. $\beta_{i,j}J^m_{[n,k]}=0$ for all $j\neq k+i$.

\begin{Example}
Let $S$ be a system with $4$ components, each of which has possible states $\{0,1,2,3\}$ such that the system is working whenever the sum of the states of the components is bigger than or equal $5$. The ideal of this system is $J^3_{[4,5]}\subseteq R={\bf k}[x,y,z,t]$ and is minimally generated by the following $40$ monomials, sorted as described before.
\begin{center}
\begin{tabular}{r|ll}
&$i=3$&$i=2$\cr
\hline
$x$&$x^3yt, x^3zt, x^3yz,x^3y^2,x^3z^2,x^3t^2$ & $x^2y^2z,x^2y^2t,x^2yz^2,x^2yt^2,x^2z^2t,x^2zt^2,x^2yzt$ \cr

$y$&$y^3zt, y^3z^2, y^3t^2, xy^3z, xy ^3t,x^2y^3$&$y^2z^2t,y^2zt^2,xy^2zt,xy^2z^2,xy^2t^2$\cr

$z$&$z^3t^2, xz^3t,yz^3t,xyz^3,x^2z^3,y^2z^3$& $xz^2t^2,yz^2t^2,xyz^2t$\cr

$t$&$xyt^3,xzt^3,yzt^3,x^2t^3,y^2t^3, z^2t^3$&$xyzt^2$\cr

\hline
\end{tabular}
\end{center}

And from this we have that $\beta_{0,5}(J^3_{[4,5]})=40$, $\beta_{1,6}(J^3_{[4,5]})=92$, $\beta_{2,7}(J^3_{[4,5]})=72$, $\beta_{3,8}(J^3_{[4,5]})=19$ and $\beta_{i,j}(J^3_{[4,5]})=0$ otherwise. Observe that, for instance, the multidegrees of the two contributions of $xz^3t$ to $\beta_{1,6}(J^3_{[4,5]})$ are $xyz^3t$ and $xz^3t^2$,  and the multidegree of its contribution to $\beta_{2,7}(J^3_{[4,5]})$ is $xyz^3t^2$  since $P_{xz^3t}=\{y,t\}$.
\end{Example}

We finish with an example of application of these systems.

\subsection{Storage problem using binary $k$-out-of-$n$ systems with multi-state components.}
Binary $k$-out-of-$n$ systems with multi-state components can be used to model storage problems in which the storage capacity is distributed among several containers. To illustrate this, let $S$ be the set of $n$ tanks in a wine cellar where grape is received in the harvesting season. Each of the tanks $T_i,\, i=1,\dots,n$ has a total capacity of $C_i$ tons and when a tractor arrives at the cellar, the staff distributes the the new coming grapes among different tanks so that the wine produced in the tanks is sufficiently homogeneous in terms of the origin of the grapes.

The filling procedure is the following: let $G$ be the number of loads of grapes in the incoming tractor (a load consists of 100Kg). We use a discrete measure of time, namely time $t$ means that we have already stored in the tanks the grapes of $t$ tractors. We denote by $l_t$ a measure of the level of the set of tanks after time $t$. We can consider $l_t$ as the average of the levels of each of the tanks, the minimum or the maximum among them.
We choose a level $l\leq\min\{C_1,\dots,C_n\}$ that we do not want to pass after storing the new coming grapes. Let $m=l-l_t$ and observe that in principle $l$ is chosen so that $m<G$. Among all the possibilities to perform the required load, we choose one randomly. Let us denote by $p^t_{i,j}$ the probability that at time $t$ the empty space in tank $T_i$ is at least $j$. We have that $ p^t_{i,0}=1$ for all $i$ and $p^t_{i,j}\geq 0$ for all $0\leq j\leq m$. If one or more of the tanks is full at time $t$ we continue with the same procedure on the remaining tanks. Our goal is to study the probability $p(l),\, l>l_t$ that we can store the $G$ new coming grape loads in the $n$ tanks so that no tank is filled beyond $l$ and assuming all tanks are already filled to level $l_t$. This situation can be modeled by a binary $G$-out-of-$n$ system with multi-state components, in which each component can be in states $\{0,\dots,m\}$.

\begin{Example} \label{ex:tank1}
Consider a cellar with $n=5$ tanks with a capacity of $15$ tons each. After a certain time $t$ the maximum level on any of the tanks is $12.5$ tons i.e. $125$ loads. A tractor arrives with $15$ loads of grapes and we want to describe how $p(l)$ behaves for $l>125$. We have modeled the probabilities $p_{i,j}$  as $p_{i,j}=1-(\frac{10}{150}j)^{3/2}$ for all $i$, and $0\leq j\leq 15$, and $p_{i,j}=0$ if $j>15$, i.e. in our case all tanks have the same probability distribution. Under these conditions we have a binary $15$-out-of-$5$ system with multi-state components such that each component can be in states $\{0,\dots,m=l-125\}$ for each $l$. Using the results in Section \ref{Sec:B_MS_KN} we have that the ideal of this system is $J^{m}_{[5,15]}$. The number of generators of this ideal, according to the formula given in Lemma \ref{lemma:formula}, gives the number of different ways to allocate the grapes meeting the requirements of the described procedure. Taking into account the probabilities of each of the tanks, we can compute the probability that we can meet the requirements using the multigraded Betti numbers as computed in Lemma \ref{lemma:formula}. We used an implementation of the formulas (\ref{eq:numgens}) and (\ref{eq:betti}) and algorithms to obtain the set of generators and Hilbert series of the corresponding ideals within the computer algebra system \texttt{Macaulay2} \cite{M2}. The results are shown in Figure \ref{fig:tank1} and Table \ref{table:tank1}, in which we also show the time (in seconds) taken for the computation of the full list of multigraded Betti numbers, from which we compute the probability in each case.

\begin{figure}[t]
\begin{center}
\includegraphics[scale=0.65]{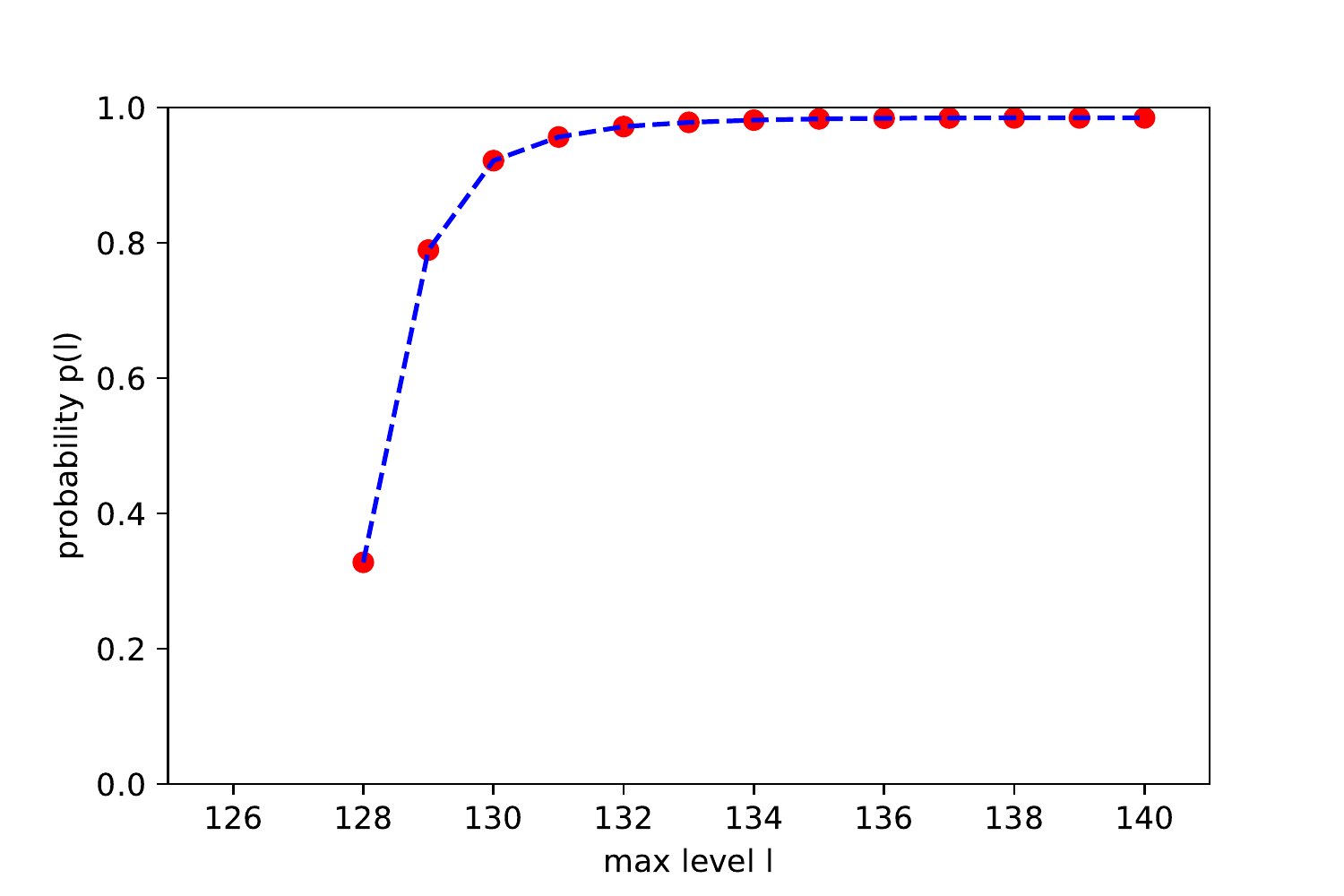}\caption{Probability that we can fill the $5$ tanks in Example \ref{ex:tank1} up to level at most $l$ for $l$ from $125$ to $140$.}\label{fig:tank1}
\end{center}
\end{figure} 

\begin{table}[h]
\begin{tabular}{r|rrr}
Level $l$&$p(l)$&$\#$ gens&time\\
\hline
$125$&$0$&--&--\\
$126$&$0$&--&--\\
$127$&$0$&--&--\\
$128$&$0.32768$&$1$&$0$\\
$129$&$0.78926$&$121$&$0.016$\\
$130$&$0.92148$&$651$&$0.126$\\
$131$&$0.95644$&$1451$&$0.532$\\
$132$&$0.97187$&$2226$&$1.140$\\
$133$&$0.97805$&$2826$&$1.594$\\
$134$&$0.98136$&$3246$&$2.057$\\
$135$&$0.98321$&$3526$&$2.274$\\
$136$&$0.98413$&$3701$&$2.470$\\
$137$&$0.98453$&$3801$&$2.799$\\
$138$&$0.98466$&$3851$&$2.821$\\
$139$&$0.98469$&$3871$&$2.834$\\
$140$&$0.98469$&$3876$&$2.821$\\
\end{tabular}
\medskip
\caption{Probabilities, number of generators and times to compute multigraded Betti numbers for the data in Example \ref{ex:tank1}}
\label{table:tank1}
\end{table}

\end{Example}

\section{Conclusions and further work}
The paper shows how to apply the authors' work on algebraic reliability to  multi-state problems. The key to the extension is to find the right monomial ideal for a suitable generalization of a $k$-out-of-$n$ system. From this the main technical problem is to find the Betti numbers which give tight reliability bounds: generalized extensions of Bonferroni bounds. Multi-unit storage, an increasingly important application, has a natural multi-state description and results are given for some simple examples.

The methods of this paper should be extendable to any multi-state systems in which there is  an identifiable state, or collections of states, which indicates a level of degradation of the system and for which extremal state may lead to the  failure of the system. There are two parts of the theory, one based the algebra and combinatorics of the system and its degradation and the other the stochastic behaviour of the system. 

Future work, therefore, will concentrate on both parts of the theory: algebraic and stochastic and, of course the interplay between the two.  We are aware that stochastic processes are indexed by {\em time} and that therefore the works should give greater priority to the time behaviour bringing in, at least, the standard models of failure. For the algebraic side each "special" state or pattern is likely to lead to different algebra, that is a different ideal or collection of ideals. On the stochastic side we are eager to allow the behaviour systems to be controlled by causal graph (network) based stochastic models, partly because they too are increasingly covered by algebraic theory,  \cite{S18}. Multi-state modeling has become increasingly part of areas such as disease modeling and emergency planning, often under a heading of component and system degradation. Future research will continue to combine Markov and other models of movement between states with the ideal theory describing the detailed structure of failure.

Finally, we should declare that the importance of energy storage, and energy networks,  is likely to lead to more work in that area. We hope also to facilitate the application to genomics, with suitable collaborations.

\section*{Acknowledgments}
The authors are partially funded by grant MTM2017-88804-P of Ministerio de Econom\'ia, Industria y Competitividad (Spain). 

\appendix
\section{A very short introduction to the algebraic method in reliability}\label{appendix:ideals}
In order to illustrate the algebraic method for system reliability analysis, we will use a simple example in which we will use all the concepts involved. A general detailed description and plenty of more elaborate examples can be found in \cite{SW09, SW10, SW11, SW12, SW15, MPSW19} where the interested reader can find full proofs of the relevant results for this approach.
 
 Our simple example is a multi-state parallel system $S$ depicted in Figure \ref{Fig:ParallelIntro} (i.e. it is a $1$-out-of-$2$ multi-state system). Let $\lbrace c_1,c_2\rbrace$ be the components of $S$ and for each component let $\mathcal{S}_1=\lbrace 0,1,2\rbrace$ and $\mathcal{S}_2=\lbrace 0,1,2,3\rbrace$ be the performance levels of $c_1$ and $c_2$ respectively. The structure function of $S$ is given by $\phi(\textbf{s})=\max\lbrace s_1,s_2\rbrace$ for $\textbf{s}=(s_1,s_2)\in\mathcal{S}_1\times\mathcal{S}_2$.
Since we have a two-component system, we can algebraically model its states in a polynomial ring with two variables, $R=\kb[x_1,x_2]$ with $\kb$ a suitable field, we can consider $\kb=\mathbb{R}$. First of all, we observe the correspondence between states of the system $S$ and monomials in $R$.
 
 \begin{figure}
     \centering
     \includegraphics[scale=0.65]{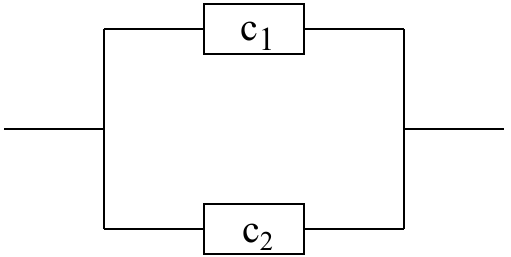}
     \caption{Multi-state parallel system}
     \label{Fig:ParallelIntro}
 \end{figure}
 
 \begin{figure}[h]
\begin{subfigure}[b]{0.48\textwidth}
\centering
    \includegraphics[scale=0.65]{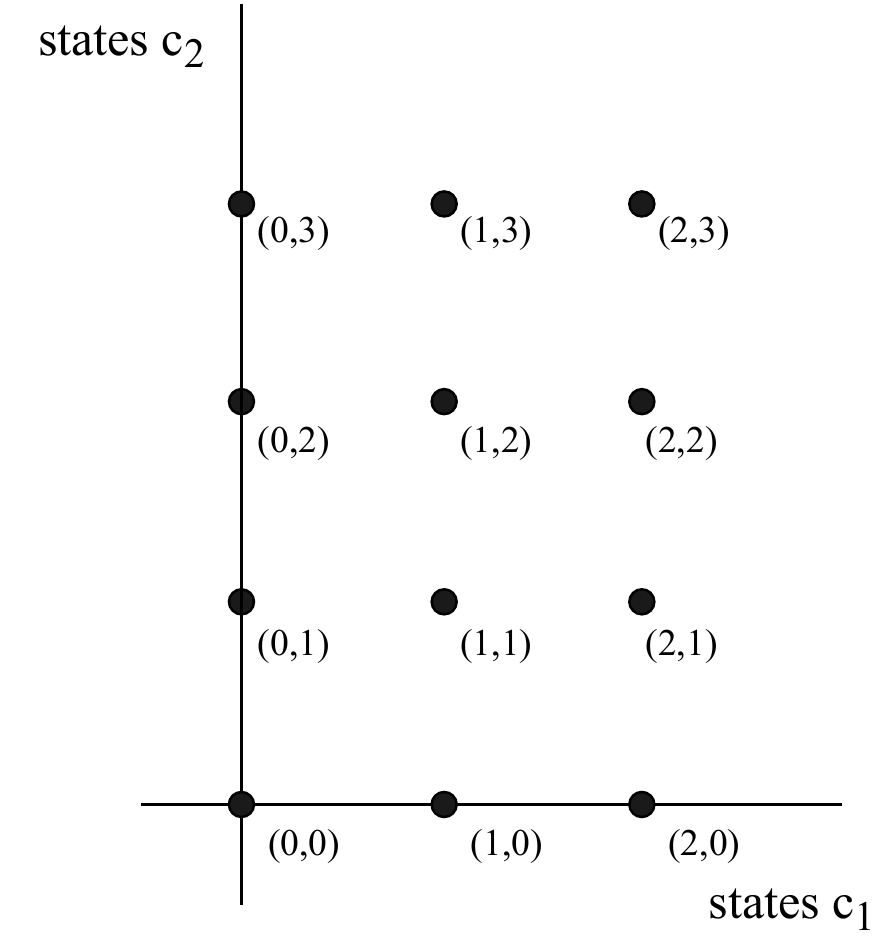}
    \caption{State space of the system $S$}
    \label{Fig:StateSpace}
 \end{subfigure} 
\begin{subfigure}[b]{0.48\textwidth}
\centering
    \includegraphics[scale=0.6]{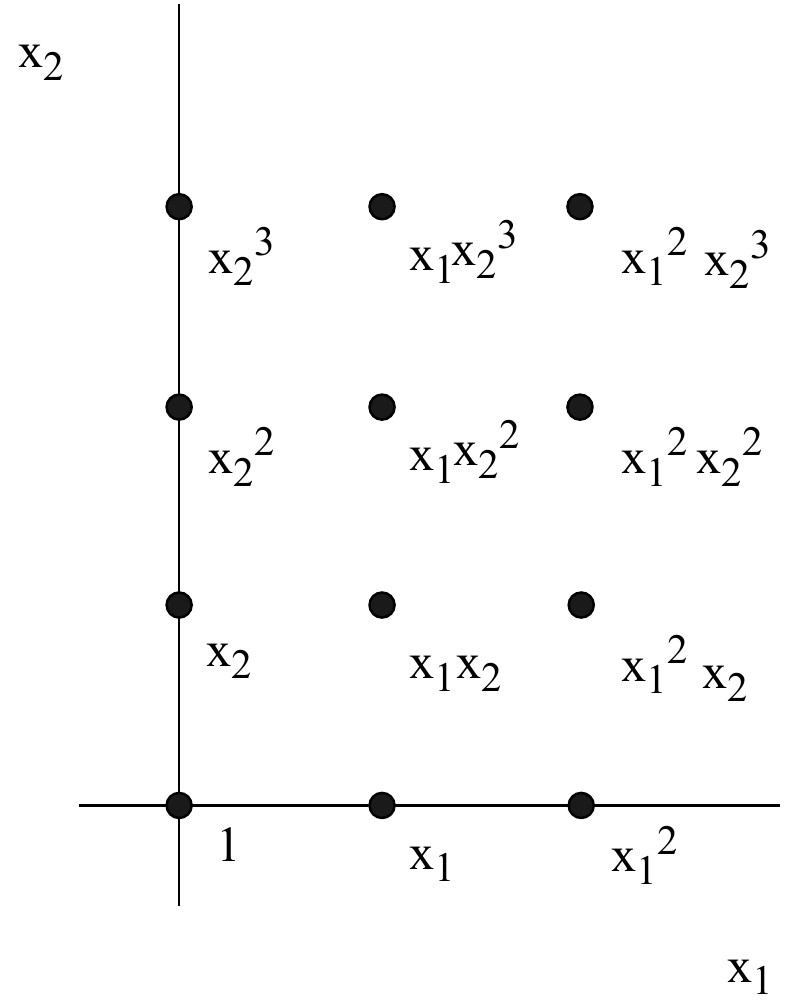}
    \caption{Equivalence between state space and monomials}
    \label{Fig:StateSpaceMonomial}
 \end{subfigure}
 \label{Fig:Relation}
 \caption{Relation between state space of the system and monomials}
 \end{figure}
 
 Figure \ref{Fig:StateSpace} shows the state space of system $S$ i.e. $\lbrace (s_1,s_2)\ :\ s_1\in\mathcal{S}_1\text{ and }s_2\in\mathcal{S}_2\rbrace$. Now, we make each state $(s_1,s_2)\in\mathcal{S}_1\times\mathcal{S}_2$ correspond with the monomial $x_1^{s_1} x_2^{s_2}$ in $R$. These monomials are represented in \ref{Fig:StateSpaceMonomial} so that the correspondence becomes clear.
  
Let us consider now the $j$-working states of $S$ for each $j$, i.e. $\mathcal{F}_{S,j}$ consists of the tuples $\textbf{s}=(s_1,s_2)$ such that $\phi(\textbf{s})\geq j,\, j\in\mathcal{S}$. We have
 \begin{align*}
     \mathcal{F}_{S,1}&=\lbrace (0,1), (0,2), (0,3), (1,0), (1,1), (1,2), (1,3), (2,0), (2,1), (2,2), (2,3)\rbrace,\\
     \mathcal{F}_{S,2}&=\lbrace (0,2), (0,3), (1,2), (1,3), (2,0), (2,1), (2,2), (2,3)\rbrace,\\
     \mathcal{F}_{S,3}&=\lbrace (0,3), (1,3), (2,3)\rbrace.\\
 \end{align*}
 
The minimal $j$-working states, denoted $\overline{\mathcal{F}}_{S,j}$ are the tuples in which if any component decreases its performance level, the performance of all the system decreases to $j'<j$. Then, we obtain
 
 \begin{align*}
     \overline{\mathcal{F}}_{S,1}&=\lbrace (0,1), (1,0)\rbrace,\\
     \overline{\mathcal{F}}_{S,2}&=\lbrace (0,2), (2,0)\rbrace,\\
     \overline{\mathcal{F}}_{S,3}&=\lbrace (0,3)\rbrace.\\
 \end{align*}
 
Having the relation between tuples of components' states and monomials into account and the coherence property of the system, we have that the $j$-working states correspond to the monomials in an ideal of $R$ which we will denote $I_{S,j}$ it is easy to see that the unique minimal monomial generating set of $I_{S,j}$, the $j$-reliability ideal of $S$ is the one corresponding to the minimal $j$-working states of the system. In our example we have that
 \begin{align*}
     I_{S,1}&=\langle x_1,x_2\rangle,\\
     I_{S,2}&=\langle x_1^2,x_2^2\rangle,\\
     I_{S,3}&=\langle x_2^3\rangle.\\
 \end{align*}
 
 That ideals are represented in Figures \ref{Fig:Ideal1}, \ref{Fig:Ideal2} and \ref{Fig:Ideal3} respectively.
 
 \begin{figure}[h]
\begin{subfigure}[b]{0.25\textwidth}
\centering
    \includegraphics[scale=0.5]{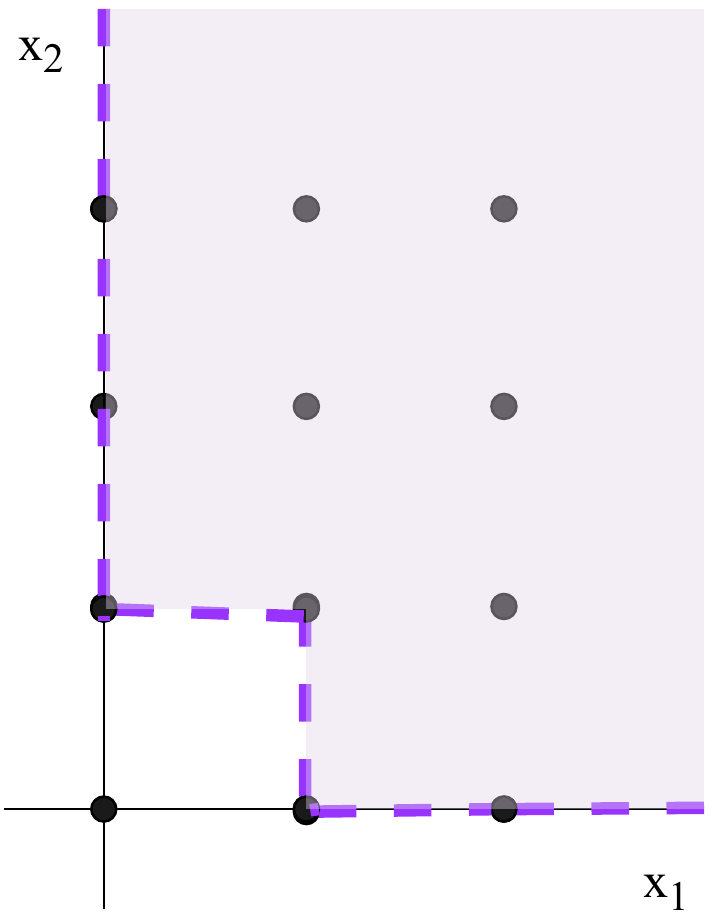}
    \caption{$1$-reliability ideal of system $S$}
    \label{Fig:Ideal1}
 \end{subfigure} 
\begin{subfigure}[b]{0.25\textwidth}
\centering
    \includegraphics[scale=0.5]{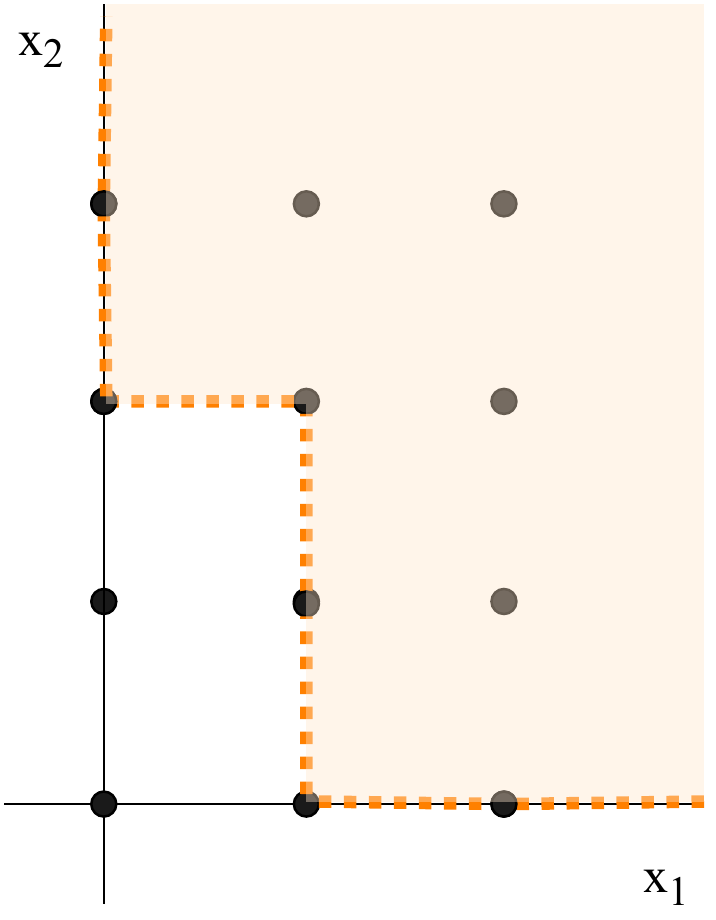}
    \caption{$2$-reliability ideal of system $S$}
    \label{Fig:Ideal2}
 \end{subfigure}
 \begin{subfigure}[b]{0.25\textwidth}
\centering
    \includegraphics[scale=0.5]{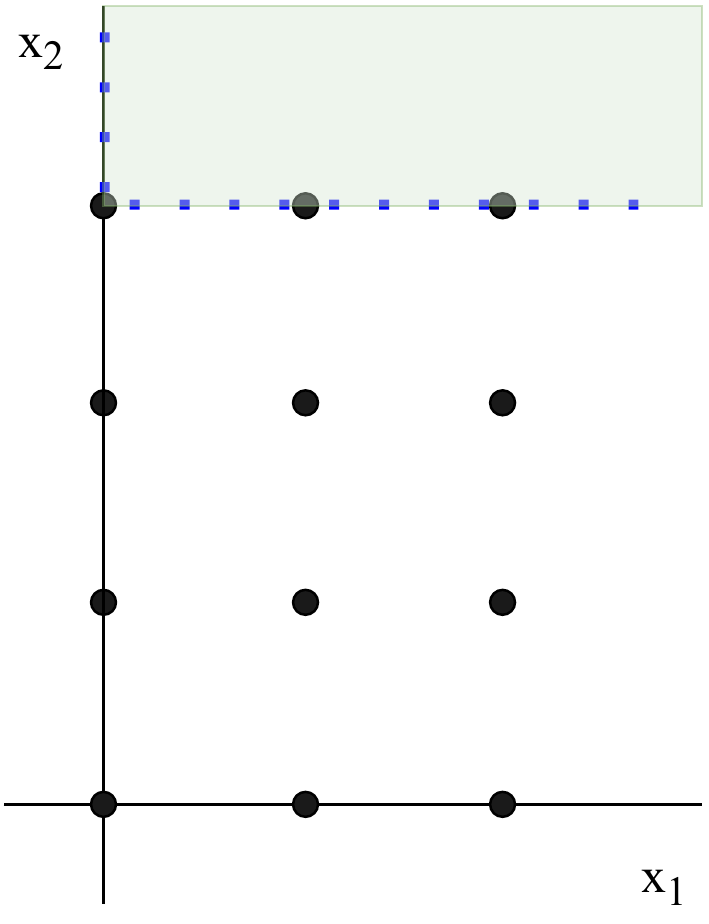}
    \caption{$3$-reliability ideal of system $S$}
    \label{Fig:Ideal3}
 \end{subfigure}
  \caption{$j$-reliability ideals for system $S$}
 \label{Fig:Ideals}
 \end{figure}
 
Observe that while the set of possible states of the system is finite, we have an infinite number of monomials in our ideal. We will deal with this issue when assigning the probability distribution to the system's components and describe its reflection at the ideal level. A powerful tool in commutative algebra to describe the structure of a monomial ideal is the Hilbert series, which is a short way to enumerate the set of monomials in a monomial ideal. It is based on the inclusion-exclusion principle and consists in adding up all the multiples of the minimal generators of the ideal, substract the multiples of the pairwise least common multiple of minimal generators, add again the multiples of the threefold least common multiples of minimal generators, and so on. There are compact ways to obtain the Hilbert series, which are beyond the scope of this paper. For full details we refer the reader to the references at the beginning of this Appendix.

Finally, to use the Hilbert function in order to obtain the $j$-reliability of the system we assign probabilities to monomials. Let's say that $p_{i,j}$ is the probability that component $i$ is in state at least $j$, we then assign to the $j$'th power of variable $i$ the probability $p_{i,j}$ and the probability of a monomial is given by the product of the probabilities assigned to its individual powers. Observe that if a variable is raised to a power that does not correspond to any state of the corresponding component, then its assigned probability is $0$ and this removes all except a finite set of monomials from the final result, except exactly those corresponding to possible states of the system.

As for this example, let us assign $p_{1,1}=0.7,\ p_{1,2}=0.3,\ p_{2,1}=0.7, p_{2,2}=0.2, p_{2,3}=0.1$.
 
 The numerator of the Hilbert series for level $1$ is $H_{I_{S,1}}=x_1+x_2-x_1x_2$. Graphically, this can be seen as:
 \begin{itemize}
     \item The ideal $\langle x_1 \rangle$ contains the monomials in the shaded area in Figure \ref{Fig:Hilbert1}
     \item The ideal $\langle x_1 \rangle$ contains the monomials in the shaded area in Figure \ref{Fig:Hilbert2}
     \item The ideal $\langle x_1x_2 \rangle$ (i.e. generated by the pairwise least common multiples of the generators of the ideal -just one such pair in this case-) contains the monomials in the shaded area in Figure \ref{Fig:Hilbert3}
 \end{itemize}  
 
 Assigning the corresponding probabilities to the monomials in  $H_{I_{S,1}}$ we obtain that the $1$-reliability for $S$ is $0.91$.
 
 Proceeding in the same way we have that the numerator of the Hilbert series for levels $2$ and $3$ are $H_{I_{S,2}}=x_1^2+x_2^2-x_1^2x_2^2$ and $H_{I_{S,3}}=x_2^3$, respectively and the $2$-reliability of $S$ is $0.38$ and the $3$-reliability of $S$ is $0.1$.
  \begin{figure}[h]
\begin{subfigure}[b]{0.29\textwidth}
\centering
    \includegraphics[scale=0.5]{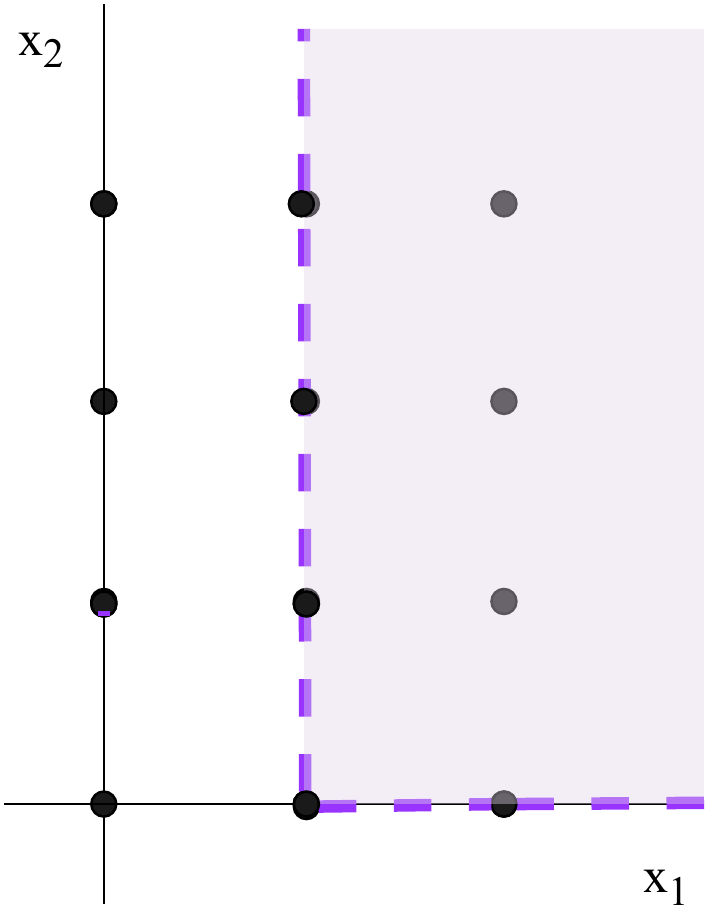}
    \caption{Monomials in $\langle x_1\rangle$}
    \label{Fig:Hilbert1}
 \end{subfigure} 
\begin{subfigure}[b]{0.29\textwidth}
\centering
    \includegraphics[scale=0.5]{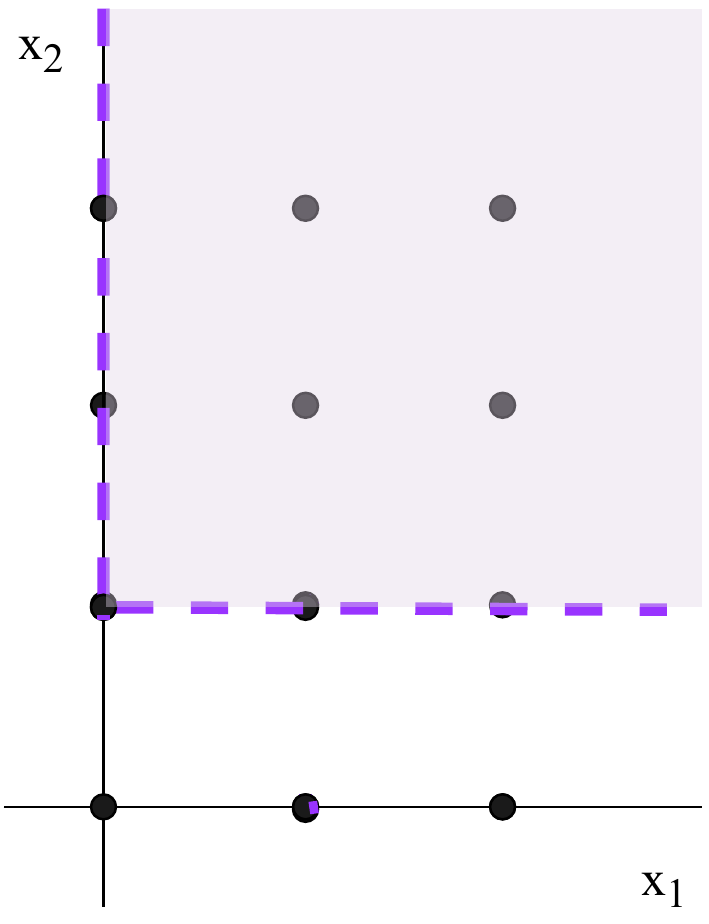}
    \caption{Monomials in $\langle x_2\rangle$}
    \label{Fig:Hilbert2}
 \end{subfigure}
 \begin{subfigure}[b]{0.29\textwidth}
\centering
    \includegraphics[scale=0.5]{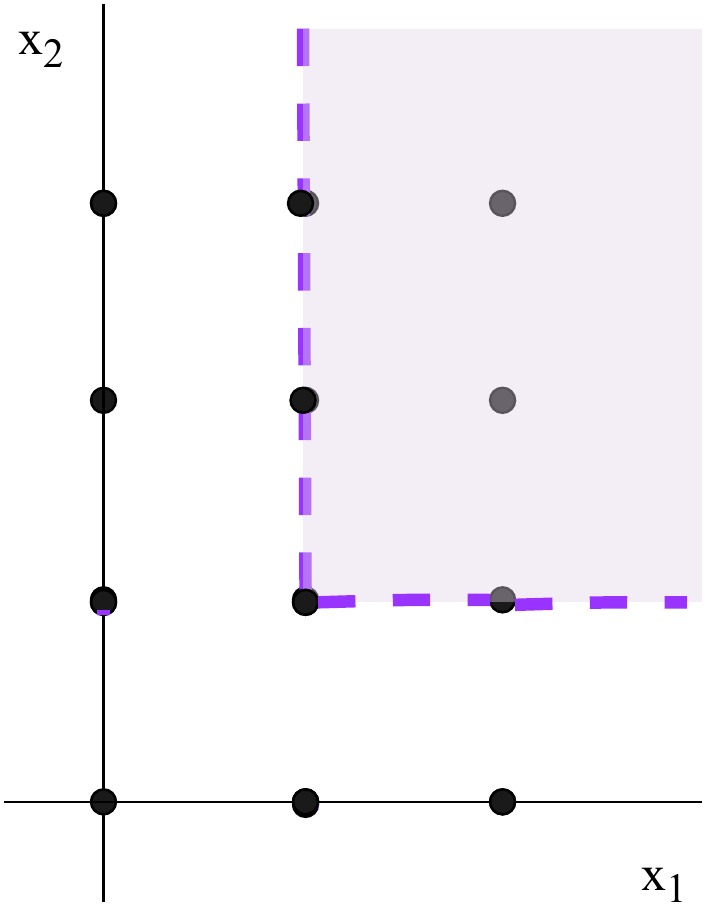}
    \caption{Monomials in $\langle x_1x_2\rangle$}
    \label{Fig:Hilbert3}
 \end{subfigure}
 \label{Fig:Areas}
 \caption{Using $H_{I_{S,1}}$ to obtain the monomials in $\langle x_1, x_2\rangle$}
 \end{figure}

\section{Mayer-Vietoris trees}\label{appendix:MVT}
Let $I\subseteq S=\kb[x_1,\dots,x_n]$ be a monomial ideal and $G=\{g_1,\dots,g_r\}$ a monomial generating set (unless otherwise stated we will always consider that $G$ is the unique minimal monomial generating set of~$I$). Fix any numbering of the elements in $G$ and let $I_i=\langle g_1\dots,g_i\rangle$ be the subideal generated by the first $i$ generators of $I$. For each $i$ we have the following exact sequence
\begin{equation}\label{eq:ses}
0\longrightarrow I_{i-1}\cap \langle g_i\rangle\stackrel{j}{\longrightarrow}I_{i-1}\oplus\langle g_i\rangle\stackrel{l}{\longrightarrow}I_i\longrightarrow 0.
\end{equation}
Assume that free resolutions $\FF'_i$ and $\widetilde{\FF}_i$ are known for $I'_i=I_{i-1}$ and $\widetilde{I_i}=I_{i-1}\cap\langle g_i\rangle$ respectively. Then, a (not necessarily minimal) resolution $\FF_i$ of $I_i$ is obtained as the mapping cone of the chain complex morphism $\psi: \widetilde{\FF_i}\longrightarrow \FF'_i$ that lifts the inclusion $j$, cf. \cite{CE95,HT02}.

Using recursively sequence (\ref{eq:ses}) on $i$ we can compute a free resolution $\FF=\FF_r$ of $I$ that is called an {\em iterated mapping cone resolution}. Observe that this process preserves (multi) degrees. The ideals involved in this process can be displayed as a binary tree. The root of this tree is $I$ and every node $J=\langle f_1\dots,f_j\rangle$ has $J'=\langle f_1,\dots,f_{j-1}\rangle$ as right child and $\widetilde{J}=J'\cap\langle f_j\rangle$ as left child. This is called a {\em Mayer-Vietoris tree} of $I$, cf. \cite{S09}.

Each node in a Mayer-Vietoris tree is assigned a position and a dimension. The root has position $1$ and dimension $0$ and the right and left children of a node with position $p$ and dimension $d$ are given positions $2p+1$ and $2p$ respectively and dimensions $d$ and $d+1$ respectively. We say that a node is {\em relevant} if it is either the root or if its position is even. The multidegrees of the minimal generators of the relevant nodes of dimension $d$ in a Mayer-Vietoris tree are then the multidegrees of the generators of the $d$-th module of the iterated mapping cone resolution $\FF$ of $I$ described by the tree. Let $\MVT(I)_{d,\mu}$ be the set of the positions of the relevant nodes of dimension $d$ of a given Mayer-Vietoris tree of $I$ having $x^\mu$ as a minimal generator. If a monomial $x^\mu$ appears only once as generator of a relevant node in the tree then if $d$ is the dimension of that node and $p$ its position let $\MVT(I)'_{d,\mu}=\{p\}$ otherwise $\MVT(I)'_{d,\mu}=\emptyset$ for all $d$. Note that if $\MVT(I)'_{d,\mu}$ is not empty, then  $\MVT(I)'_{d,\mu}= \MVT(I)_{d,\mu}$. Since the minimal free resolution of $I$ is a subresolution of $\FF$ we have that for any Mayer-Vietoris tree the following result holds \cite{S09}.

\begin{Proposition}\label{prop:MVTbounds}
For any Mayer-Vietoris tree of $I$
\[
\#\MVT(I)'_{d,\mu}\leq\beta_{d,\mu}(I)\leq\#\MVT(I)_{d,\mu}.
\]

\end{Proposition}

%

The generators of the relevant nodes of $\MVT(I)$ provide upper and lower bounds for the Betti numbers of the ideal without actually computing the resolution. These bounds can be improved using several criteria and are sharp in several families of ideals, see \cite{S09} for details. A simple useful criterion is the following:
\begin{Proposition}
Let $\mu$ be a multidegree such that there are generators of multidegree $\mu$ in relevant nodes of $\MVT(I)$ of dimensions $d_1\dots d_k$ such that no two of them are consecutive, then 
\[
\beta_{d_i,\mu}(I)=\#\MVT(I)_{d_i,\mu}.
\]
\end{Proposition}

We say that two generators $e^{(i)}_\sigma$ and $e^{(i-1)}_\tau$ of $\FF$ with the same multidegree form a {\em reduction pair} if the coefficient of $e^{(i-1)}_\tau$ in $\varphi(e^{(i)}_\sigma)$ is a non-zero scalar, i.e. if we can reduce $\FF$ by deleting $e^{(i)}_\sigma$ and $e^{(i-1)}_\tau$ and adjusting $\varphi_i$. Reduction pairs appear only in {\em compatible} nodes. Let $J$ and $J'$ two nodes of $\MVT(I)$ whose first common ancestor is $K$ and such that $J$ is a descendant of $\widetilde{K}$ and $J'$ is a descendant of $K'$ we say $J$ and $J'$ are compatible if $\dim(J)-\dim(\widetilde{K})=\dim(J')-\dim(K')$. Compatibility of $J$ and $J'$ can be read from the binary expression of their positions. We can therefore ensure that $\beta_{d,\mu}(I)$ is bigger than or equal to the number of generators of multidegree $\mu$ in relevant nodes of dimension $d$ in $\MVT(I)$ such that they have no compatible generator. Hence, if there are no compatible generators, we obtain the Betti numbers of~$I$ directly from $\MVT(I)$.

\begin{Example}
Let us consider Mayer-Vietoris trees of ideals of consecutive linear $k$-out-of-$n$:G systems. Theses systems work if at least $k$ consecutive components of the $n$ components of the system work. The corresponding ideal is of the form $I_{k,n}=\langle x_1\cdots x_k,\dots, x_{n-k+1}\cdots x_n\rangle$. The Mayer-Vietoris tree of the ideal of the consecutive linear $2$-out-of-$5$ system, taking as pivot always the last generator, is

\begin{center}
 \begin{tikzpicture}[scale=1]
 \tikzstyle{level 1}=[sibling distance=6.5cm]
 \tikzstyle{level 2}=[sibling distance=3.5cm]
 \node{$(1,0)$ $x_1x_2,x_2x_3,x_3x_4,x_4x_5$}
 child{ node{$(2,1)$ $x_1x_2x_4x_5,x_3x_4x_5$}
 	child{node{$(4,2)$ $x_1x_2x_3x_4x_5$}}
 	child{node[color=black!50!white]{$(5,1)$ $x_1x_2x_4x_5$}}}
 child{ node[color=black!50!white]{$(3,0)$ $x_1x_2,x_2x_3,x_3x_4$}
 	child{node{$(6,1)$ $x_2x_3x_4$}}
 	child{node[color=black!50!white]{$(7,0)$ $x_1x_2,x_2x_3$}
	child{node{$(14,1)$ $x_1x_2x_3$}}
	child{node[color=black!50!white]{$(15,0)$ $x_1x_2$}
	   }}};
 \end{tikzpicture}
 \end{center}

From this tree we obtain that $\beta_{0,2}(I_{2,5})=4$, $\beta_{1,3}(I_{2,5})=3$, $\beta_{1,4}(I_{2,5})=1$ and $\beta_{2,5}(I_{2,5})=1$. Moreover, the numerator of the Hilbert series of this ideal is 
\[
HN_{I_{2,5}}=(x_1x_2+x_2x_3+x_3x_4+x_4x_5)-(x_1x_2x_4x_5+x_3x_4x_5+x_2x_3x_4+x_1x_2x_3)+x_1x_2x_3x_4x_5
\]
As one can see, the node at position $3$ of $MVT(I_{k,n})$ is just $I_{k,n-1}$ so the contribution of this branch of the tree is just a smaller case of the same kind. The analysis of the other branch of the tree is also straightforward and we can easily come up with a recursive formula for the Betti numbers of $I_{k,n}$ as was shown in \cite{SW09}. Using this kind of reasoning on Mayer-Vietoris trees we come out with recursive formulas like (\ref{ms-k-out-of-n-formula}).
\end{Example}


\end{document}